\documentclass[12pt]{amsart}
\usepackage{amsmath,amssymb}
\usepackage{amsfonts}
\usepackage{amsthm}
\usepackage{latexsym}
\usepackage{graphicx}
\usepackage{caption}
\def\tr{{\rm tr}}

\def\R{\mathbb{R}}
\def\C{\mathbb{C}}

\def\vv<#1>{\langle#1\rangle}

\def\XXint#1#2{\setbox0=\hbox{$#1{#2}{\int}$}{#2}\kern-.5\wd0 }

\def\XXint#1#2#3{{\setbox0=\hbox{$#1{#2#3}{\int}$}
     \vcenter{\hbox{$#2#3$}}\kern-.5\wd0}}

\def\vv<#1>{\left\langle#1\right\rangle}
\def\e{\epsilon}
\def\Aut{{\rm Aut}}
\def\Diff{{\rm Diff}}
\def\Aff{{\rm Aff}}
\newtheorem{thm}{Theorem}[section]

\newtheorem{lem}{Lemma}[section]

\theoremstyle{definition}
\newtheorem{defn}{Definition}[section]
\theoremstyle{remark}

\newtheorem{rem}{Remark}[section]

\numberwithin{equation}{section}

\begin{document}

\title{Affine Self-similar solutions of the affine curve shortening flow I: The degenerate case}
\author{Chengjie Yu$^1$}
\address{Department of Mathematics, Shantou University, Shantou, Guangdong, 515063, China}
\email{cjyu@stu.edu.cn}
\author{Feifei Zhao}
\address{Department of Mathematics, Shantou University, Shantou, Guangdong, 515063, China}
\email{14ffzhao@stu.edu.cn}

\thanks{$^1$Research partially supported by the Yangfan project from Guangdong Province and NSFC 11571215.}

\renewcommand{\subjclassname}{%
  \textup{2010} Mathematics Subject Classification}
\date{}
\begin{abstract}
In this paper, we consider affine self-similar solutions for the affine curve shortening flow in the Euclidean plane.  We obtain the equations of all affine self-similar solutions up to affine transformations and solve the equations or give descriptions of the solutions for the degenerate case. Some new special solutions for the affine curve shortening flow are found.
\end{abstract}
\maketitle\markboth{Yu \& Zhao}{Affine self-similar solutions}
\section{introduction}
A generalized curve shortening flow (GCSF) in $\R^2$ is a family of curve $\hat X(t,u)$ satisfying the following evolution equation:
\begin{equation}\label{equ-GCSF}
\hat X_t=|\hat k|^{\sigma-1}\hat k{\hat N},
\end{equation}
where ${\hat N}$ is the unit normal vector field on $\hat X$, $\hat k$ is the curvature of $\hat X$ with respect to $\hat N$ and $\sigma$ is a positive constant.  When $\sigma=1$, $\hat X$ is called a curve shortening flow (CSF). The behaviors of CSF were extensively studied in the past decades. See for example \cite{AL,An,Ga,GH,Gr,Hu}. When $\sigma=\frac{1}{3}$, $\hat X$ is called an affine curve shortening flow (ACSF) since the equation \eqref{equ-GCSF} is affine invariant in this case. It was first introduced by Sapiro and Tannenbaum \cite{ST} as an analogue of the CSF for affine geometry. For behaviors of the ACSF, see for example \cite{And1,ST,AST}. For the behaviors of the GCSF, see for example \cite{And3}. For a systematic investigation the topic, see the book \cite{CZ} by Chou and Zhu.

In this paper, we will consider self-similar solutions of the ACSF.  In the classical sense, a family of curves $\hat X(t,u)$ is said to be a self-similar solution of \eqref{equ-GCSF} if $\hat X(t,u)$  is a solution of \eqref{equ-GCSF} up to reparametrization, and moreover, the curve $\hat X(t,\cdot)$ at each time $t$ can be transformed from the initial curve $X(\cdot):=\hat X(0, \cdot)$ by a conformal map on $\R^2$ preserving orientation which is a composition of a dilation, a rotation and a translation of $\R^2$. Intuitively and generally speaking,  a self-similar solution is a special solution of \eqref{equ-GCSF} such that for each time $t$, the curve $\hat X(t,\cdot)$ is similar to the initial curve $X$. The similarity is measured by a group action on $\R^2$.  In the classical sense, the similarity is measured by the automorphism group $\Aut(\C)$ when viewing $\R^2$ as $\C$ naturally. Generally, we can measure similarity by any group action on $\R^2$. Therefore, we can extend the meaning of self-similar solutions in the following sense.
\begin{defn}\label{def-SSS}
Let $G$ be a Lie group acting on $\R^2$ and $X:\R\to \R^2$ be a curve. We say that $X$ generates a self-similar solution with respect to $G$ if there is a curve $g(t)$ in $G$ with $t\in [0,T)$ for some $T>0$ and $g(0)=e$ such that
\begin{equation}\label{equ-SSS}
\hat X(t,u)=g(t).X(u)
\end{equation}
 is a solution of \eqref{equ-GCSF} up to reparametrization. We simply call $X$ a $G$-self-similar solution of \eqref{equ-GCSF}. The curve $g(t)$ in $G$ is called the family  of self-similar actions acting on $X$.
\end{defn}
Noting the form of \eqref{equ-SSS}, we know that finding self-similar solutions in the sense of Definition \ref{def-SSS} is something like using the method of separating variables to obtain special solutions of a partial diffenretial equation. If $G$ is chosen to be the group $\Diff(\R^2)$ of diffeomorphisms acting on $\R^2$ in Definition \ref{def-SSS}, it is clear that any embedded curve developing a GCSF is a $\Diff(\R^2)$-self-similar solution of GCSF. According to the definition, self-similar solutions in the classical sense are $\Aut(\C)$-self-similar solutions.

$\Aut(\C)$-self-similar solutions for CSF were studied in \cite{AL,Al,Is,Ur} and finally completely classified by Halldorsson \cite{Ha1}. The reason for choosing the group $\Aut(\C)$ is that the symmetric group of CSF (see \cite{CL}) acting on $\R^2$ conformally. The list of $\Aut(\C)$-self-similar solutions for CSF is as follows (see \cite{Ha1}):
\begin{enumerate}
\item Translation: Only the Grim Reaper curve.
\item Expansion: A 1-D family of curves each of which is properly embedded and asymptotic to a boundary of a cone. These curves were first found by Ishimura \cite{Is} and Urbas \cite{Ur}.
\item Contraction: A 1-D family of curves each of which is contained in an annulus and consists of identical excursions between its two boundaries. The closed ones in the family were completely classified by Abresch and Langer \cite{AL} and are called the Abresch-Langer curves.
\item Rotatation: A 1-D family of curves each of which is properly embedded and spirals out to infinity. The ying-yang spirals first introduced by Altschuler  \cite{Al} are contained in the family. A different discussion of this case can also be found in \cite{CZ}.
\item Rotation and expansion: A 2-D family of curves each of which is properly embedded and spirals out to infinity.
\item Rotation and contraction: A 2-D family of curves each of which has one end asymptotic to a circle and the other is either asymptotic to the same circle or spirals out to infinity.
\end{enumerate}
The last two cases (5) and (6) were first discovered by Halldorsson \cite{Ha1}. Classification of self-similar solutions of the GCSF with respect to homothetic contractions was done by Andrews (see \cite{And1,And2} ). Nien-Tsai \cite{NT} considered translation self-similar solutions for much more general curve curvature flows including the GCSF. Similar classification for the mean curvature flow in $\R^{1,1}$ was also carried out in \cite{Ha2} by using a similar technique as in \cite{Ha1}.

Denote the group of affine transformations on $\R^2$ as $\Aff(\R^2)$. As computed in \cite{CL}, the symmetric group of the ACSF acts on $\R^2$ in the same way as $\Aff(\R^2)$. So, for the ACSF, instead of considering $\Aut(\C)$-self-similar solutions, it is more natural to consider $\Aff(\R^2)$-self-similar solutions. We call such self-similar solutions of the ACSF affine self-similar solutions.

In this paper, by using a similar technique as in \cite{Ha1}, we show that an affine self-similar solution $X$ of the ACSF must satisfy the equation:
\begin{equation}\label{equ-SSS-0}
\vv<BX+C,N_X>=k_{X}^\frac13
\end{equation}
where $B\in \R^{2\times 2}$ and  $C\in \R^2$. Here $N_X$ is the normal vector field of $X$ and $k_X$ is the curvature of $X$ with respect to $N_X$. Then, by an affine transformation, we are able to simplify \eqref{equ-SSS-0} and completely classify affine self-similar solutions of the ACSF into several cases. More precisely, we have the following result.
 \begin{thm}\label{thm-SSS-classification}
Let $X:\R\to \R^2$ be an affine self-similar solution of the ACSF satisfying \eqref{equ-SSS-0}. Then, there is an affine transformation $Y=QX+H$ of $X$ with $Q\in GL(2,\R)$ and $H\in \R^2$ such that $Y$ satisfies one of the following equations.
\begin{enumerate}
\item The case: $\det B=0$. We call this the degenerate case.
\begin{enumerate}
\item If $B=0$, then
\begin{equation}\label{equ-1-a}
\vv<\left(\begin{matrix}0\\1
\end{matrix}\right),N_Y>=k_Y^{\frac{1}{3}}.
\end{equation}
This is the translation case.
\item If $B\neq 0$ has a positive eigenvalue, and $C\in R(B)$ where $R(B)$ means the range of $B$ , then
\begin{equation}\label{equ-1-b}
\vv<\left(\begin{matrix}0&0\\0&1\end{matrix}\right)Y,N_Y>
=k_Y^{\frac{1}{3}}.
\end{equation}
This is the case of expansion in a single direction.
\item If $B\neq 0$ has a positive eigenvalue and $C\not\in R(B)$, then
\begin{equation}\label{equ-1-c}
\vv<\left(\begin{matrix}0&0\\0&1\end{matrix}\right) Y+\left(\begin{matrix}1\\0\end{matrix}\right),N_Y>
=k_Y^{\frac{1}{3}}.
\end{equation}
This is the case of expansion in a direction and translation in another direction.
\item If $B\neq 0$ has a negative eigenvalue, and $C\in R(B)$, then
\begin{equation}\label{equ-1-d}
\vv<\left(\begin{matrix}0&0\\0&-1\end{matrix}\right) Y,N_Y>
=k_Y^{\frac{1}{3}}.
\end{equation}
This is the case of contraction in a single direction.
\item If $B\neq 0$ has a negative eigenvalue and $C\not\in R(B)$, then
\begin{equation}\label{equ-1-e}
\vv<\left(\begin{matrix}0&0\\0&-1\end{matrix}\right) Y+\left(\begin{matrix}1\\0\end{matrix}\right),N_Y>
=k_Y^{\frac{1}{3}}.
\end{equation}
This is the case of contraction in a direction and translation in another direction.
\item If $B\neq 0$ has two zero eigenvalues and $C\in R(B)$, then
    \begin{equation}\label{equ-1-f}
\vv<\left(\begin{matrix}0&1\\0&0\end{matrix}\right)Y,N_Y>
=k_Y^{\frac{1}{3}}.
\end{equation}
We call this the skew steady case.
\item If $B\neq 0$ has two zero eigenvalues and $C\not\in R(B)$, then
\begin{equation}\label{equ-1-g}
\vv<\left(\begin{matrix}0&1\\0&0\end{matrix}\right)Y+
\left(\begin{matrix}0\\1\end{matrix}\right),N_Y>
=k_Y^{\frac{1}{3}}
\end{equation}
We call this the case of skew steady with translation.
\end{enumerate}
\item The case: $\det B\neq 0$. We call this the nondegenerate case. Let $\Delta=(\tr B)^2-4\det B$ which is the discriminant of the characteristic polynomial of $B$.
\begin{enumerate}
\item If $\Delta<0$ and $\tr B=0$, then
\begin{equation}\label{equ-2-a}
\vv<JY,N_Y>=k_Y^{\frac{1}{3}}
\end{equation}
where $J=\left(\begin{matrix}0&-1\\1&0
\end{matrix}\right)$. This is the case of rotation.
\item If $\Delta<0$ and $\tr B>0$, then
\begin{equation}\label{equ-2-b}
\vv<Y+aJY,N_Y>=k_Y^{\frac{1}{3}}.
\end{equation}
with $a>0$. This is the case of rotation and expansion.
\item If $\Delta<0$ and $\tr B<0$, then
\begin{equation}\label{equ-2-c}
\vv<-Y+aJY,N_Y>=k_Y^{\frac{1}{3}}.
\end{equation}
with $a>0$. This is the case of rotation and contraction.
\item If $\Delta=0$ and $B$ is diagonalizable with a double positive eigenvalue, then
\begin{equation}\label{equ-2-d}
\vv<Y,N_Y>=k_Y^{\frac{1}{3}}.
\end{equation}
This is the case of expansion.
\item If $\Delta=0$ and $B$ is diagonalizable with a double negative eigenvalue, then
\begin{equation}\label{equ-2-e}
-\vv<Y,N_Y>=k_Y^{\frac{1}{3}}.
\end{equation}
This is the case of contraction.
\item If $\Delta=0$ and $B$ is not diagonalizable with a double positive eigenvalue, then
\begin{equation}\label{equ-2-f}
\vv<\left(\begin{matrix}1&1\\0&1\end{matrix}\right)Y,N_Y>
=k_Y^{\frac{1}{3}}.
\end{equation}
We call this the case of skew expansion.
\item If $\Delta=0$ and $B$ is not diagonalizable with a double negative eigenvalue, then
\begin{equation}\label{equ-2-h}
\vv<\left(\begin{matrix}-1&1\\0&-1\end{matrix}\right)Y,N_Y>
=k_Y^{\frac{1}{3}}.
\end{equation}
We call this the case of skew contraction.
\item If $\Delta>0$, $\det B>0$ and $B$ has two different positive eigenvalues, then
    \begin{equation}\label{equ-2-i}
    \vv<\left(\begin{matrix}a&0\\0&1
    \end{matrix}\right)Y,N_Y>=k_Y^{\frac{1}{3}}
    \end{equation}
    with $0<a<1$. This is the case of expansion in two directions with different ratios.
\item If $\Delta>0$, $\det B>0$ and $B$ has two different negative eigenvalues, then
    \begin{equation}\label{equ-2-j}
    -\vv<\left(\begin{matrix}a&0\\0&1
    \end{matrix}\right)Y,N_Y>=k_Y^{\frac{1}{3}}
     \end{equation}
    with $0<a<1$. This is the case of contraction in two directions with different ratios.
\item If $\Delta>0$, $\det B<0$, then
    \begin{equation}
    \vv<\left(\begin{matrix}-a&0\\0&1
    \end{matrix}\right)Y,N_Y>=k_Y^{\frac{1}{3}}
    \end{equation}
    with $a>0$. This is the case of contraction in a direction and expansion in another direction.
\end{enumerate}
\end{enumerate}
\end{thm}

The second part of this paper is to solve the equations or give descriptions of the curves in the degenerate case of Theorem \ref{thm-SSS-classification}. The nondegenerate case will be handled in another paper. We summarize the conclusion as follows:
\begin{enumerate}
\item Translation: Parabolas. This was first found in \cite{COT} where they called this the affine Grim Reaper.
\item Expansion in a single direction: Hyperbolas, graph of a convex (concave) function with only one minimum (maximum) point and graph of an increasing (decreasing) function with only one inflection point. For details, see Theorem \ref{thm-1-b}. The hyperbolas in this case were also found in \cite{CL}.
\item Expansion in a direction and translation in another direction: There are four kinds of curves. See Figure \ref{Fig-1-c-1}, Figure \ref{Fig-1-c-2}, Figure \ref{Fig-1-c-3} and Figure \ref{Fig-1-c-4}. For details, see Theorem \ref{thm-1-c}.
\item Contraction in a single direction: Graph of a periodic function. For details, see Theorem \ref{thm-1-d}.
\item Contraction in a direction and translation in another direction: Graph of a function $f(x)$ oscillating around the $x$-axis or a combination of graphs of two functions oscillating around the $x$-axis. For details, see Theorem \ref{thm-1-e}.
\item Skew steady: Graphs of some quintic functions. This solution was also found in \cite{COT}.
\item Skew steady with translation: Parabolas and curves given in the parametrization form in \eqref{equ-1-g-x} and \eqref{equ-1-g-y}. This is the case of parabolic scooper in \cite{COT}.
\end{enumerate}
Lines are also trivial solutions of the cases (1)--(6) which are not given in the list above. For details, see Theorem \ref{thm-1-a}-- Theorem \ref{thm-1-g}.

 We would like to mention that, in \cite{CL}, Chou and Li systematically computed the group $S(\sigma)$ of invariant transformations of \eqref{equ-GCSF} viewed as a partial differential equation with three variables (one time dimension and two spatial dimensions), an optimal system of 1-D subalgebra with respect to self-adjoint action, and classified the invariant solutions of \eqref{equ-GCSF} with respect to a one-parameter subgroup of $S(\sigma)$. According to the definition of self-similar solutions, self-similar solutions and invariant solutions with respect to a one-parameter subgroup of $S(\sigma)$ are not a priori identical. In fact, the two sets of special solutions have nonempty intersection and are not equal. Moreover, Chou and Li \cite{CL} only derived the equations of invariant solutions and did not give descriptions of all the invariant solutions such as in the work of Halldorsson \cite{Ha1}. Similar computation as in \cite{CL} for the centro-affine curve shortening flow was carried out in \cite{WYW}. Invariant solutions for higher dimensional curve shortening flows was studied in \cite{AAAW}.

The organization of the remaining parts of this paper is as follows. In Section 2, we will derive \eqref{equ-SSS-0} and prove Theorem \ref{thm-SSS-classification}. In Section 3, we will handle the degenerate case of Theorem \ref{thm-SSS-classification}.

\noindent\emph{Acknowledgement.} The authors would like to thank Professor Xiaoliu Wang for helpful discussions on self-similar solutions and invariant solutions.

\section{Equations of  Affine self-similar solutions}
In this section, we derive the equation of affine self-similar solutions of the ACSF by a similar technique as in \cite{Ha1}.

\begin{lem}\label{lem-SSS-eq}
An immersed curve $X:\R\to \R^2$ is an affine self-similar solution of ACSF if and only if it satisfies the following equation:
\begin{equation}\label{equ-SSS0}
\vv<BX+C,N>=k^\frac{1}{3},
\end{equation}
where $B\in \R^{2\times 2}$  and $C\in \R^2$. Moreover, the family $(A(t),H(t))$ of self-similar actions on $X$ satisfying \eqref{equ-SSS0} is given by
\begin{equation}\label{equ-A0}
A(t)=\left\{\begin{array}{ll}\exp(tB)&\tr B=0\\
\exp\left(\frac{3}{2\tr B}\ln \left(1+\frac23\tr B\cdot t\right)B\right)&\tr B\neq 0
\end{array}\right.
\end{equation}
and
\begin{equation}\label{equ-H0}
H(t)=\int_0^t(\det A)^{-\frac{2}3}A(t)dt\cdot C.
\end{equation}
\end{lem}
\begin{proof}
Let $X:\R\rightarrow\R^{2}$ be an affine self-similar solution of the ACSF.
Then, by Definition \ref{def-SSS}, there is a curve $(A(t),H(t))$ in $\Aff(\R^2)$ for $t\in [0,T)$ such that
\begin{equation}\label{equ-hat-X}
\hat{X}(t,u)=A(t)X(u)+H(t),
\end{equation}
is a solution \eqref{equ-GCSF} up to reparametrization with $\sigma=1/3$. That is to say, $\hat X$ satisfies:
\begin{equation}\label{equ-GCSF-2}
\left\langle\hat X_t,\hat{N}\right\rangle=\hat k^{\frac13}.
\end{equation}
Here $A(t)\in GL_+(2,\R)$, and $H(t)\in \R^2$. Moreover, $A(0)=I_2$ and $H(0)=0$.

For the curve $X(u)$, the unit tangent vector of $X$ is $T=\frac{X_u}{\|X_u\|}$. The unit normal vector is $N=JT$ where
\begin{equation}\label{equ-J}
J=\left(\begin{matrix}
0&-1\\1&0
\end{matrix}\right).
\end{equation}
The curvature of $X$ with respect to $N$ is
\begin{equation}
k=\frac{\det(X_u,X_{uu})}{\|X_u\|^3}.
\end{equation}
Similarly for the curve $\hat X$, by \eqref{equ-hat-X}, the unit tangent vector of $\hat X$ is
\begin{equation}\label{equ-hat-T}
\hat T=\frac{\hat X_u}{\|\hat X_u\|}=\frac{AX_u}{\|AX_u\|}=\frac{\|X_u\|}{\|AX_u\|}AT,
\end{equation}
the unit normal vector field of $\hat X$ is
\begin{equation}\label{equ-hat-N}
\hat N=J\hat T=\frac{\|X_u\|}{\|AX_u\|}JAT=\frac{\|X_u\|}{\|AX_u\|}JAJ^tN=\frac{\|X_u\|\det A}{\|AX_u\|}\left(A^{-1}\right)^tN,
\end{equation}
and the curvature of $\hat X$ with respect to $\hat N$ is
\begin{equation}\label{equ-hat-k}
\hat k=\frac{\det(\hat X_{u},\hat X_{uu})}{\|\hat X_u\|^3}=\frac{\det A\det(X_u,X_{uu})}{\|AX_u\|^3}=\frac{\|X_u\|^3\det A}{\|AX_u\|^3}k.
\end{equation}
Substituting \eqref{equ-hat-N} and \eqref{equ-hat-k} into \eqref{equ-GCSF-2}, we have
\begin{equation}\label{equ-SSS-1}
(\det A)^{\frac23}\langle A^{-1}A'X(u)+A^{-1}H',N\rangle=k^\frac13.
\end{equation}
Noting that the RHS of \eqref{equ-SSS-1} is independent of $t$ while the LHS of \eqref{equ-SSS-1} is depending on $t$ and letting $t=0$, we obtain
\begin{equation}\label{equ-SSS-2}
\vv<BX+C,N>=k^\frac13
\end{equation}
where $B=A'(0)$ and $C=H'(0)$. This means that an affine self-similar solution of the ACSF must satisfy \eqref{equ-SSS-2}. Conversely, for a curve satisfying \eqref{equ-SSS-2}, by letting
\begin{equation}\label{equ-A-H}
\left\{\begin{array}{l}(\det A)^{\frac23}A^{-1}A'=B\\
(\det A)^{\frac23}A^{-1}H'=C\\
A(0)=I_2\\
H(0)=0.
\end{array}\right.
\end{equation}
We can obtain the self-similar action $(A(t),H(t))$ on $X$ that produces a solution of \eqref{equ-GCSF}.

By the first and third equation of \eqref{equ-A-H}, we have
\begin{equation}\label{equ-A}
A(t)=\exp\left(\int_0^t(\det A)^{-\frac{2}{3}}dt \cdot B\right).
\end{equation}
Taking determinant of \eqref{equ-A}, we have
\begin{equation}
\det A=\exp\left(\int_0^t(\det A)^{-\frac{2}{3}}dt \cdot \tr B\right).
\end{equation}
So, $\det A$ satisfies
\begin{equation}\label{equ-det-A}
\left\{\begin{array}{l}\frac{d}{dt}\det A=\tr B (\det A)^{\frac{1}{3}}\\
\det A(0)=1.
\end{array}\right.
\end{equation}
Solving \eqref{equ-det-A}, we obtain
\begin{equation}\label{equ-det-A-1}
\det A=\left(1+\frac23\tr B\cdot  t\right)^\frac32.
\end{equation}
Substituting this into \eqref{equ-A}, we have
\begin{equation}\label{equ-A-2}
A(t)=\left\{\begin{array}{ll}\exp(tB)&\tr B=0\\
\exp\left(\frac{3}{2\tr B}\ln \left(1+\frac23\tr B\cdot t\right)B\right)&\tr B\neq 0.
\end{array}\right.
\end{equation}
Then, by substituting \eqref{equ-A-2} and \eqref{equ-det-A-1} into \eqref{equ-A-H}, one can obtain the expression of $H$.
This completes the proof of the lemma.
\end{proof}
\begin{rem}
Note that when $\tr B<0$, by the expression of $A(t)$ in the lemma, the maximal existence time of the ACSF is $T=-\frac{3}{2\tr B}$.
\end{rem}
Because of the affine invariance of the ACSF, we have also the following relation of an affine self-similar solution and its affine transformation.

\begin{lem}\label{lem-SSS-aff}
Let $X:\R\to \R^2$ be an affine self-similar solution satisfying
\begin{equation}
\vv<BX+C,N_X>=k_X^{\frac{1}{3}}
\end{equation}
where $N_X$ and $k_X$ are the normal vector field and curvature of $X$ respectively. Let $Q\in GL(2,\R)$ and $Y=Q^{-1}X$. Then $Y$ satisfies
\begin{equation}
(\det Q)^\frac{2}3\vv<Q^{-1}BQY+Q^{-1}C,N_Y>=k_Y^{\frac{1}{3}}.
\end{equation}
\end{lem}
\begin{proof}
By \eqref{equ-hat-N} and \eqref{equ-hat-k}, we have
\begin{equation}
N_X=\frac{\|Y_u\|\det Q}{\|QY_u\|}(Q^{-1})^tN_Y
\end{equation}
and
\begin{equation}
k_X=\frac{\|Y_u\|^3\det Q}{\|QY_u\|^3}k_Y.
\end{equation}
Substituting these into the equation of $X$, we complete the proof of the lemma.
\end{proof}

We are now ready to prove Theorem \ref{thm-SSS-classification}.
\begin{proof}[Proof of Theorem \ref{thm-SSS-classification}]We will discuss the existence of the affine transformation case by case.
\begin{enumerate}
\item[(1-a)] Since $C\neq 0$, there is clearly a non-degenerate matrix $Q$ such that
    \begin{equation}
    (\det Q)^\frac{2}{3}Q^{-1}C=(0,1)^t.
    \end{equation}
    Let $Y=Q^{-1}X$. By Lemma \ref{lem-SSS-aff}, we get the conclusion.
\item[(1-b)]  Note that there is a $H\in \R^2$ such that $BH=C$. Let $Y_1=X+H$. It is clear that $Y_1$ satisfies:
    \begin{equation}
    \vv<BY_1,N_{Y_1}>=k_{Y_1}^{\frac{1}{3}}.
    \end{equation}
    Moreover, there is a $Q_1\in GL(2,\R)$ with $\det Q_1=\pm 1$ such that
\begin{equation}
Q_1^{-1}BQ_1=\left(\begin{matrix}0&0\\0&\lambda
\end{matrix}\right).
\end{equation}
with $\lambda>0$. Let $Y=\lambda^\frac{3}{4}Q_1^{-1}Y_1$. Then, by Lemma \ref{lem-SSS-aff}, we get the conclusion.
\item[(1-c)] Let $Y_1=\lambda^\frac{3}{4}Q_1^{-1}X$ where $Q_1$ is the same as in the proof of (1-b). Then, as in (1-b), $Y_1$ satisfies:
    \begin{equation}
    \vv<\left(\begin{matrix}0&0\\0&1
\end{matrix}\right)Y_1+C_1,N_{Y_1}>=k_{Y_1}^{\frac{1}{3}}
    \end{equation}
    where $C_1$ is not in the range of $\left(\begin{matrix}0&0\\0&1
\end{matrix}\right)$. So $C_1=(a,b)^t$ with $a\neq 0$. Let $Y_2=Y_1+(0,b)^t$. Then, it is clear that
\begin{equation}
    \vv<\left(\begin{matrix}0&0\\0&1
\end{matrix}\right)Y_2+\left(\begin{matrix}a\\0
\end{matrix}\right),N_{Y_2}>=k_{Y_2}^{\frac{1}{3}}.
    \end{equation}
Finally, let $Y=\left(\begin{matrix}\frac1{a}&0\\0&a
\end{matrix}\right)Y_2$.
Then,  by Lemma \ref{lem-SSS-aff}, we get the conclusion.
\item[(1-d)] The proof is the same as that of (1-b).
\item[(1-e)] The proof is the same as that of (1-c).
\item[(1-f)] Let $H\in \R^2$ be such that $BH=C$ and let $Y_1=X+H$. Then
\begin{equation}
    \vv<BY_1,N_{Y_1}>=k_{Y_1}^{\frac{1}{3}}.
    \end{equation}
 Let $Q_1\in GL(2,R)$ with $\det Q_1=\pm 1$ be such that
\begin{equation}
Q_1^{-1}BQ_1=\left(\begin{matrix}0&1\\0&0
\end{matrix}\right)
\end{equation}
and $Y=Q_1^{-1}Y_1$. By Lemma \ref{lem-SSS-aff}, we obtain the conclusion.

\item[(1-g)]
 Let $Q_1\in GL(2,R)$ with $\det Q_1=\pm1$ be such that
\begin{equation}
Q_1^{-1}BQ_1=\left(\begin{matrix}0&1\\0&0
\end{matrix}\right)
\end{equation}
and $Y_1=Q_1^{-1}X$. By Lemma \ref{lem-SSS-aff},
\begin{equation}
\vv<\left(\begin{matrix}0&1\\0&0
\end{matrix}\right)Y_1+C_1,N_{Y_1}>=k_{Y_1}^{\frac{1}{3}}
\end{equation}
with $C_1=(a,b)^t$ is not in the range of $\left(\begin{matrix}0&1\\0&0
\end{matrix}\right)$ and hence $b\neq 0$. Let $Y_2=Y_1+(0,a)^t$. Then,
\begin{equation}
\vv<\left(\begin{matrix}0&1\\0&0
\end{matrix}\right)Y_2+\left(\begin{matrix}0\\b
\end{matrix}\right),N_{Y_2}>=k_{Y_2}^{\frac{1}{3}}.
\end{equation}
Let $Y=\left(\begin{matrix}b^\frac53&0\\0&b^\frac13
\end{matrix}\right)Y_2$. Then,  by Lemma \ref{lem-SSS-aff}, we obtain the conclusion.
\item[(2-a)] Let $Y_1=X+B^{-1}C$. Then, we have
\begin{equation}
    \vv<BY_1,N_{Y_1}>=k_{Y_1}^{\frac{1}{3}}.
    \end{equation}

In this case, $B$ has eigenvalues $\pm ci$ with $c>0$. So, there is a $Q_1\in GL(2,\R)$ with $\det Q_1=\pm1$, such that
$$Q_1^{-1}BQ_1=cJ.$$
Here $J$ is the matrix defined in \eqref{equ-J}. Let $Y=c^{\frac{3}4}Q_1^{-1}Y_1$. Then, by Lemma \ref{lem-SSS-aff}, we get the conclusion.
\item[(2-b)] In this case, $B$ has eigenvalues $b\pm ci$ with $b>0$ and $c>0$. So, there is a $Q_1\in GL(2,\R)$ with $\det Q_1=\pm 1$ such that $$Q_1^{-1}BQ_1=bI+cJ.$$
    Letting $Y=b^{\frac{3}4}Q_1^{-1}Y_1$ where $Y_1$ is the same as in (2-a) and by Lemma \ref{lem-SSS-aff}, we get the conclusion.
\item[(2-c)] The proof is the same that of (2-b).
\item[(2-d)] The proof is similar with that of (2-a).
\item[(2-e)] The proof is similar with that of (2-a).
\item[(2-f)] In this case, there is a $Q_1\in GL(2,\R)$ with $\det Q_1=\pm 1$ such that
\begin{equation}
Q_1^{-1}BQ_1=\left(\begin{matrix}\lambda&1\\0&\lambda
\end{matrix}\right)
\end{equation}
with $\lambda>0$. Let $Y=\left(\begin{matrix}\lambda^\frac54&0\\0&\lambda^\frac{1}{4}
\end{matrix}\right)Q_1^{-1} Y_1$ with $Y_1$ the same as in (2-a). Then, by Lemma \ref{lem-SSS-aff}, we get the conclusion.
\item[(2-g)] The proof is similar with that of (2-f).
\item[(2-h)] The proof is similar with that of (2-a).
\item[(2-i)] The proof is similar with that of (2-a).
\item[(2-j)] The proof is similar with that of (2-a).
\end{enumerate}
\end{proof}
\section{The case: $\det B=0$.}
In this section, we solve the equations or give descriptions  of the curves in the cases (1-a)--(1-g) in Theorem \ref{thm-SSS-classification}. We will discuss case by case.

\begin{thm}[Translation]\label{thm-1-a}The solutions of \eqref{equ-1-a} are  the parabolas $y=\frac{1}{2}x^2+C_1x+C_2$ and the lines $x=C$ where $C_1$, $C_2$ and $C$ are arbitrary constants .
\end{thm}
Although the general case was discussed in \cite{CZ} and this case was also solved in \cite{COT}, we will also present the proof here for completeness.
\begin{proof}[Proof of Theorem \ref{thm-1-a} ] First, suppose that $Y$ is locally of the form $(x,y(x))$. Then,
\begin{equation}
N=(1+y_x^2)^{-\frac12}(-y_x,1)^t
\end{equation}
and
\begin{equation}
k=(1+y_x^2)^{-\frac32}y_{xx}.
\end{equation}
Substituting these into \eqref{equ-1-a}, we have
\begin{equation}
y_{xx}=1.
\end{equation}
So, $y=\frac{1}{2}x^2+C_1x+C_2$.

On the other hand, when assume that $Y$ is locally of the form $(x(y),y)$, we have
\begin{equation}
N=(1+x_y^2)^{-\frac{1}2}(-1,x_y)^t
\end{equation}
and
\begin{equation}
k=-(1+x_y^2)^{-\frac32}x_{yy}.
\end{equation}
Substituting these into \eqref{equ-1-a}, we have
\begin{equation}
x_{yy}=-x_y^3.
\end{equation}
It is clear that $x_y=0$ and hence $x=C$ is a solution of the equation.
\end{proof}
\begin{thm}[Expansion in a single direction]\label{thm-1-b}Solutions of \eqref{equ-1-b} are the following curves:
\begin{enumerate}
\item The lines $x=C$ and $y=0$ with $C$ any constant.
\item The hyperbolas $(x-C)y=\pm\sqrt 2$ with $C$ any constant.
\item $(x-C_2)^2-\left(\int_{\small{\sqrt[4]{2C_1}}}^y\frac{1}{\sqrt{\frac{1}{2}\xi^4-C_1}}d\xi\right)^2=0$ with $y\geq \sqrt[4]{2C_1}$ which is the graph of a convex function with only one minimum point at $x=C_2$. Here $C_1$ is a positive constant.
\item $(x-C_2)^2-\left(\int^{\small{-\sqrt[4]{2C_1}}}_{y}\frac{1}{\sqrt{\frac{1}{2}\xi^4-C_1}}d\xi\right)^2=0$ with $y\leq- \sqrt[4]{2C_1}$ which is the graph of a concave function with only one maximum point $x=C_2$. Here $C_1$ is a positive constant.
\item $x=\int_0^y\frac{1}{\sqrt{\frac12\xi^4+C_1}}d\xi+C_2$ with $C_1>0$ which is  the graph of an increasing function with only one inflection point $x=C_2$.
\item $x=-\int_0^y\frac{1}{\sqrt{\frac12\xi^4+C_1}}d\xi+C_2$ with $C_1>0$ which is  the graph of a decreasing function with only one inflection point $x=C_2$.
\end{enumerate}
\end{thm}
\begin{proof}
First assume that $Y$ is locally of the form $(x(y),y)$, similarly as in the proof of Theorem \ref{thm-1-a}, we have
\begin{equation}
x_{yy}+y^3 x_y^3=0.
\end{equation}
This ODE has a special solution $x\equiv C$. Moreover, when $x_y\neq0$, the ODE can be rewritten as
\begin{equation}
(x_y^{-2})_y=2y^3.
\end{equation}
So,
\begin{equation}
x_y=\pm\frac{1}{\sqrt{\frac{1}2y^4+C}}.
\end{equation}
From this,  we obtain the hyperbolas in (2) when $C=0$,  obtain the expression in (3) and (4) when $C<0$,  and   obtain the expression in (5) and (6) when $C>0$.

On the other hand, when locally written $Y$ in  the form $(x,y(x))$, similarly as in the proof of Theorem \ref{thm-1-a}, we have
\begin{equation}\label{equ-1-b-y_xx}
y_{xx}=y^3.
\end{equation}
Clearly, $y=0$ is a solution of \eqref{equ-1-b-y_xx}.
\end{proof}
\begin{thm}[Expansion in a direction and translation in another direction]\label{thm-1-c} Up to translations and reflection with respect to the $x$-axis, the solutions of \eqref{equ-1-c} are the following five kinds of curves:
\begin{enumerate}
\item The line $y=0$.
\item A curve formed by the graphs of two functions $y=f_1(x)$ and $y=f_2(x)$ with $x\in [0,+\infty)$ where $f_1$ and $f_2$ satisfy the following properties (see Figure \ref{Fig-1-c-1}):
    \begin{enumerate}
    \item $f_1(0)=f_2(0)>0$, $f'_1(0)=+\infty$ and $f'_2(0)=-\infty$;
    \item $f_1(x)>f_2(x)$ for any $x\in (0,+\infty)$;
    \item $f_1$ is increasing with only one inflection point;
    \item $f_2$ is convex function with only one minimum point with positive minimum.
    \end{enumerate}
\item A curve formed by the graphs of two functions $y=f_1(x)$ and $y=f_2(x)$ with $x\in [0,+\infty)$ where $f_1$ and $f_2$ satisfy the following properties (see Figure \ref{Fig-1-c-2}):
    \begin{enumerate}
    \item $f_1(0)=f_2(0)>0$, $f'_1(0)=+\infty$ and $f'_2(0)=-\infty$;
    \item $f_1$ is increasing with only one inflection point;
    \item $f_2$ is a decreasing convex function and $f_2(x)\to 0$ as $x\to +\infty$.
    \end{enumerate}
\item A curve formed by the graphs of two functions $y=f_1(x)$ and $y=f_2(x)$ with $x\in [0,+\infty)$ where $f_1$ and $f_2$ satisfy the following properties (see Figure \ref{Fig-1-c-3}):
    \begin{enumerate}
    \item $f_1(0)=f_2(0)\geq0$, $f'_1(0)=+\infty$ and $f'_2(0)=-\infty$;
    \item $f_1$ is increasing with only one inflection point;
    \item $f_2$ is decreasing with only one inflection point.
    \end{enumerate}
    Moreover, when $f_1(0)=f_2(0)=0$, the curve is symmetric with respect to the $x$-axis. That is, $f_2=-f_1$.
\item The graph of an increasing convex function $f$ on $(-\infty,\infty)$ with $f(x)\to 0$ as $x\to-\infty$. (See Figure \ref{Fig-1-c-4}.)
\end{enumerate}

\end{thm}

\begin{figure*}[!htb]
	\centering
	\begin{minipage}{0.45\textwidth}
		\centering
		\includegraphics[width=1\linewidth, height=0.2\textheight]{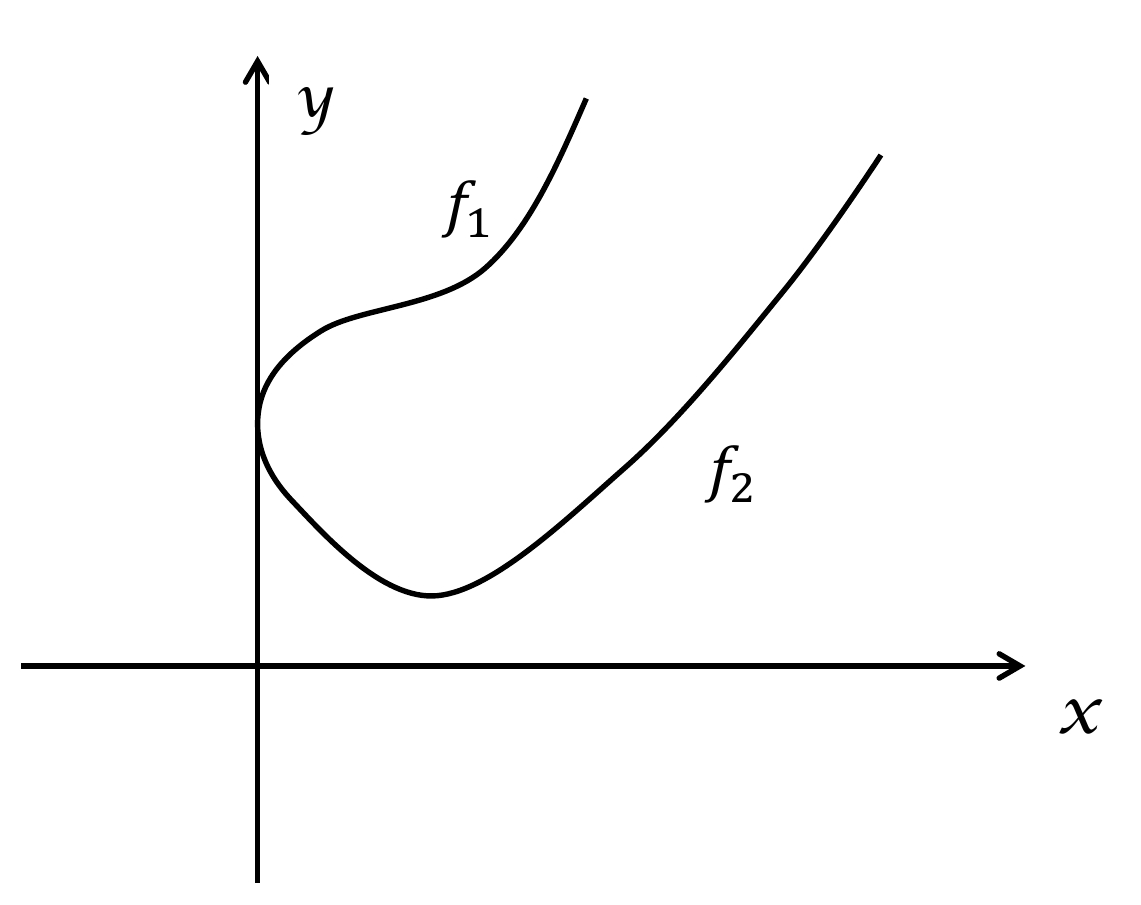}
		\caption{The curve for case (2) of Theorem \ref{thm-1-c}}\label{Fig-1-c-1}			 \end{minipage}%
		\begin{minipage}{0.45\textwidth}
		\centering
		\includegraphics[width=1\linewidth, height=0.2\textheight]{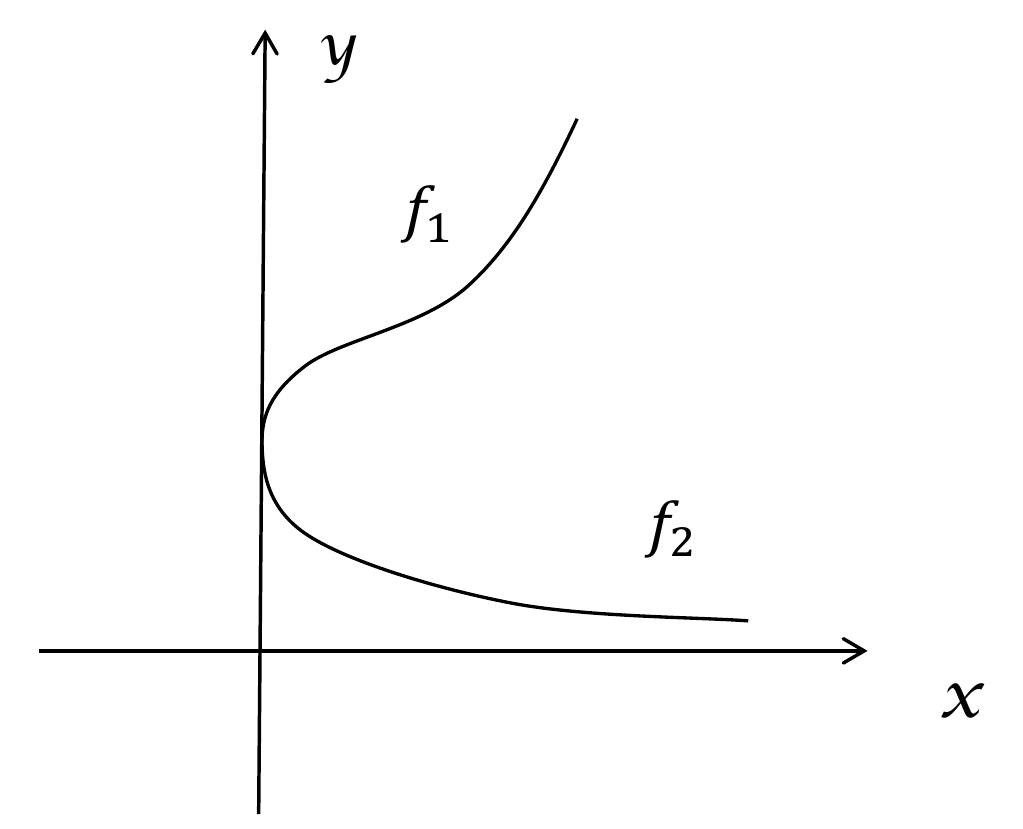}
		\caption{The curve for case (3) of Theorem \ref{thm-1-c}}\label{Fig-1-c-2}
	\end{minipage}
\end{figure*}
\begin{figure*}[!htb]
	\centering
\begin{minipage}{0.45\textwidth}
		\centering
		\includegraphics[width=1\linewidth, height=0.2\textheight]{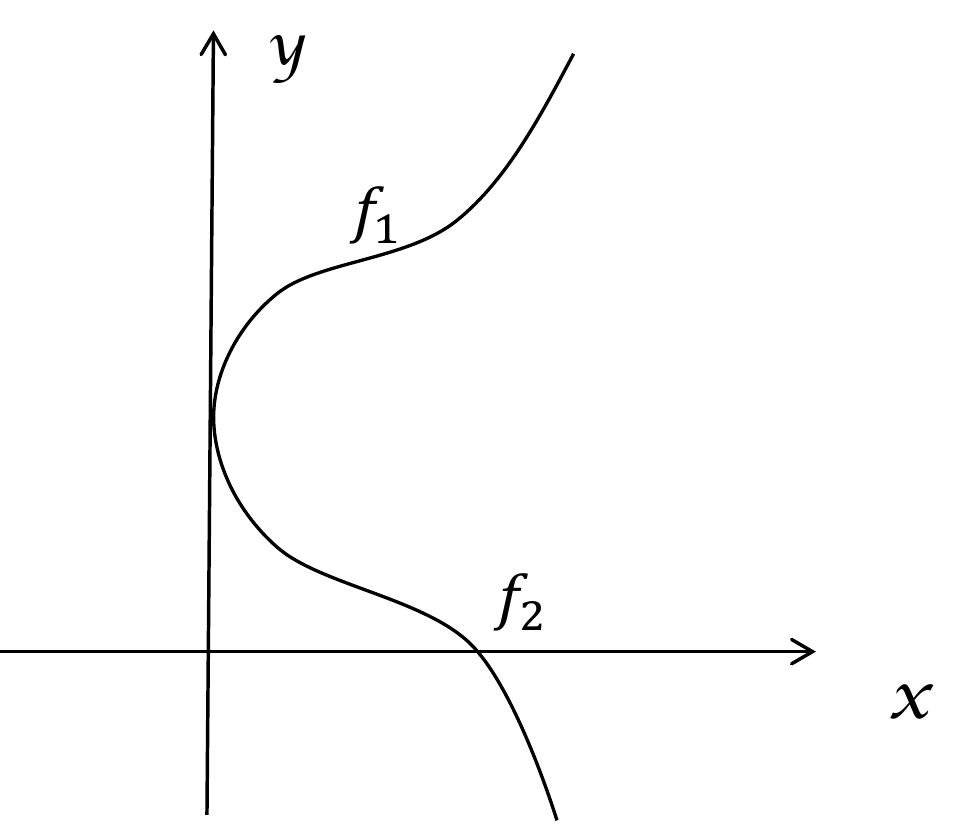}
		\caption{{ The curve for case (4) of Theorem \ref{thm-1-c}}}\label{Fig-1-c-3}
	\end{minipage}
\begin{minipage}{0.45\textwidth}
		\centering
		\includegraphics[width=1\linewidth, height=0.2\textheight]{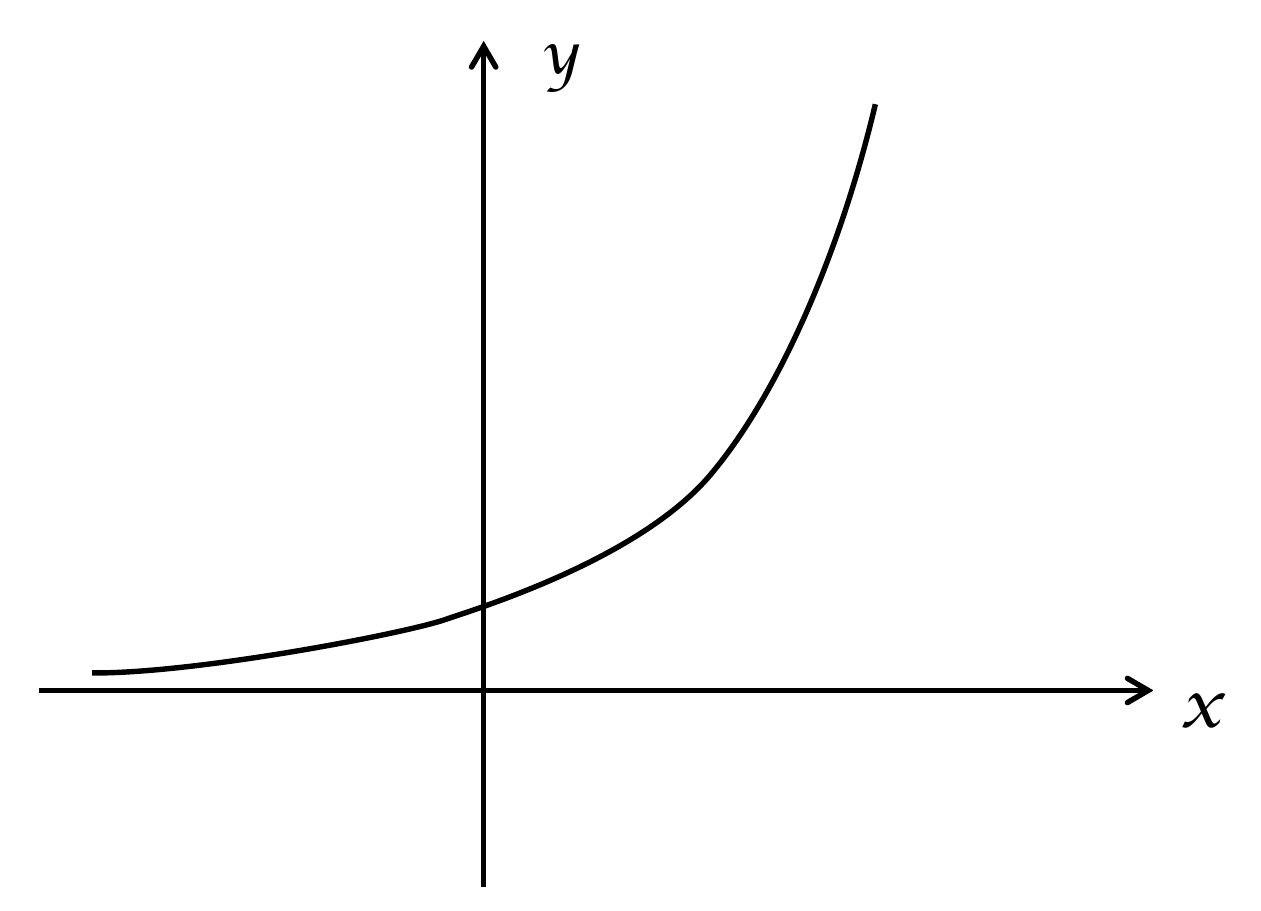}
		\caption{The curve for case (5) of Theorem \ref{thm-1-c}}\label{Fig-1-c-4}
	\end{minipage}	
\end{figure*}
%\begin{figure}
%  % Requires \usepackage{graphicx}
%  \includegraphics[width=3cm]{Fig-1-c-1.pdf}\\
%  \caption{The curve for case (2) of Theorem \ref{thm-1-c}}\label{Fig-1-c-1}
%\end{figure}
%\begin{figure}
%  % Requires \usepackage{graphicx}
%  \includegraphics[width=3cm]{Fig-1-c-2.pdf}\\
%  \caption{The curve for case (3) of Theorem \ref{thm-1-c}}\label{Fig-1-c-2}
%\end{figure}
%\begin{figure}
%  % Requires \usepackage{graphicx}
%  \includegraphics[width=3cm]{Fig-1-c-3.pdf}\\
%  \caption{The curve for case (4) of Theorem \ref{thm-1-c}}\label{Fig-1-c-3}
%\end{figure}
%\begin{figure}
%  % Requires \usepackage{graphicx}
%  \includegraphics[width=3cm]{Fig-1-c-4.pdf}\\
%  \caption{The curve for case (5) of Theorem \ref{thm-1-c}}\label{Fig-1-c-4}
%\end{figure}
\begin{proof}
Locally written $Y$ in the form $(x,y(x))$, similarly as in the proof of Theorem \ref{thm-1-a}, we have
\begin{equation}\label{equ-1-c-y_xx}
y_{xx}=(y-y_x)^3.
\end{equation}
First, note that if $y=y(x)$ is a solution of the equation, then so is $y=-y(x)$ and moreover $y=0$ is a solution of the equation.

Let $u=y$ and $v=y_x$. Then, by \eqref{equ-1-c-y_xx}, we have the following planar dynamical system:
\begin{equation}\label{equ-1-c-u-v}
\left\{\begin{array}{l}u_x=v\\
v_x=(u-v)^3.
\end{array}\right.
\end{equation}
Because the system is centrosymmetric with respect to the origin, we only need to analyze the trajectories of the dynamical system on the right half plane.

The dynamical system \eqref{equ-1-c-u-v} has a singularity at the origin which corresponding to the line $y=0$.

A trajectory of \eqref{equ-1-c-u-v} must satisfy
\begin{equation}\label{equ-1-c-u-v-2}
\frac{d u}{dv}=\frac{v}{(u-v)^3}.
\end{equation}
For those trajectories $(u(v),v)$ passing through $(u_0,v_0)$ with $u_0-v_0<0$ and $v_0>0$, by \eqref{equ-1-c-u-v-2}, we know that $u(v)$ is decreasing when  $v>v_0$. So
\begin{equation}
u(v)-u_0=\int_{v_0}^v\frac{\xi}{(u(\xi)-\xi)^3}d\xi\geq\int_{v_0}^v\frac{\xi}{(u_0-\xi)^3}d\xi\geq\int_{v_0}^\infty\frac{\xi}{(u_0-\xi)^3}d\xi.
\end{equation}
 This implies that $u(v)$ has a lower bound for $v>v_0$ and hence the trajectory will decrease to a constant as $v\to +\infty$.

For those trajectories $(u,v)$ passing through $(u_0,v_0)$ with $u_0-v_0>0$ and $v_0>0$, $u$ will increase to $+\infty$. Otherwise the trajectory will stay in the triangle region $\{(u,v)\ |\ C\geq u\geq v\geq 0\}$ with some constant $C$ and this is impossible. Moreover, $v$ will also increase to $+\infty$. Otherwise, suppose that $v\leq C$ with $C$ the supremum of $v$. Then, one has $$\liminf_{u\to+\infty}\left|\frac{dv}{du}\right|= 0.$$
 However
\begin{equation}
\frac{dv}{du}=\frac{(u-v)^3}{v}\to+\infty
\end{equation}
as $u\to +\infty$ since $v\to C$ as assumed before.

Furthermore, for those trajectories $(u,v)$ passing through $(u_0,v_0)$ with $u_0-v_0>0$ and $v_0>0$, we want to show that as $u\to +\infty$, $x\to+\infty$. Let $w=u-v$. Then,
\begin{equation}
w_x=v-w^3
\end{equation}
Consider the dynamical system:
\begin{equation}\label{equ-1-c-w-v}
\left\{\begin{array}{l}w_x=v-w^3\\
v_x=w^3.
\end{array}\right.
\end{equation}
It is clear that for those trajectories passing through $(w_0,v_0)$ with $v_0>0$ and $w_0\geq v_0^\frac{1}{3}$, the trajectory will decrease to pass though the curve $w=v^\frac13$, stay below it and become increasing in $w$ (see Figure \ref{Fig-1-c-w-v}). For those trajectories passing through $(w_0,v_0)$ with $v_0^\frac13>w_0>0$, the trajectory will be increasing in $w$ and stay below the curve $w=v^\frac13$. In summary, for sufficiently large $v$, for instance $v>C$, one has $0<w<v^\frac{1}3$. Now, suppose that $u(x_0)=u_0$ and $v(x_0)=v_0$. Then, by \eqref{equ-1-c-w-v},
\begin{equation}
x-x_0=\int_{v_0}^vw^{-3}(\xi)d\xi\geq \int_{C}^{v}\xi^{-1}d\xi=\ln v-\ln C.
\end{equation}
So, as $v\to +\infty$, $x\to +\infty$.

\begin{figure*}[!htb]
	\centering
\begin{minipage}{0.45\textwidth}
		\centering
		\includegraphics[width=1\linewidth, height=0.3\textheight]{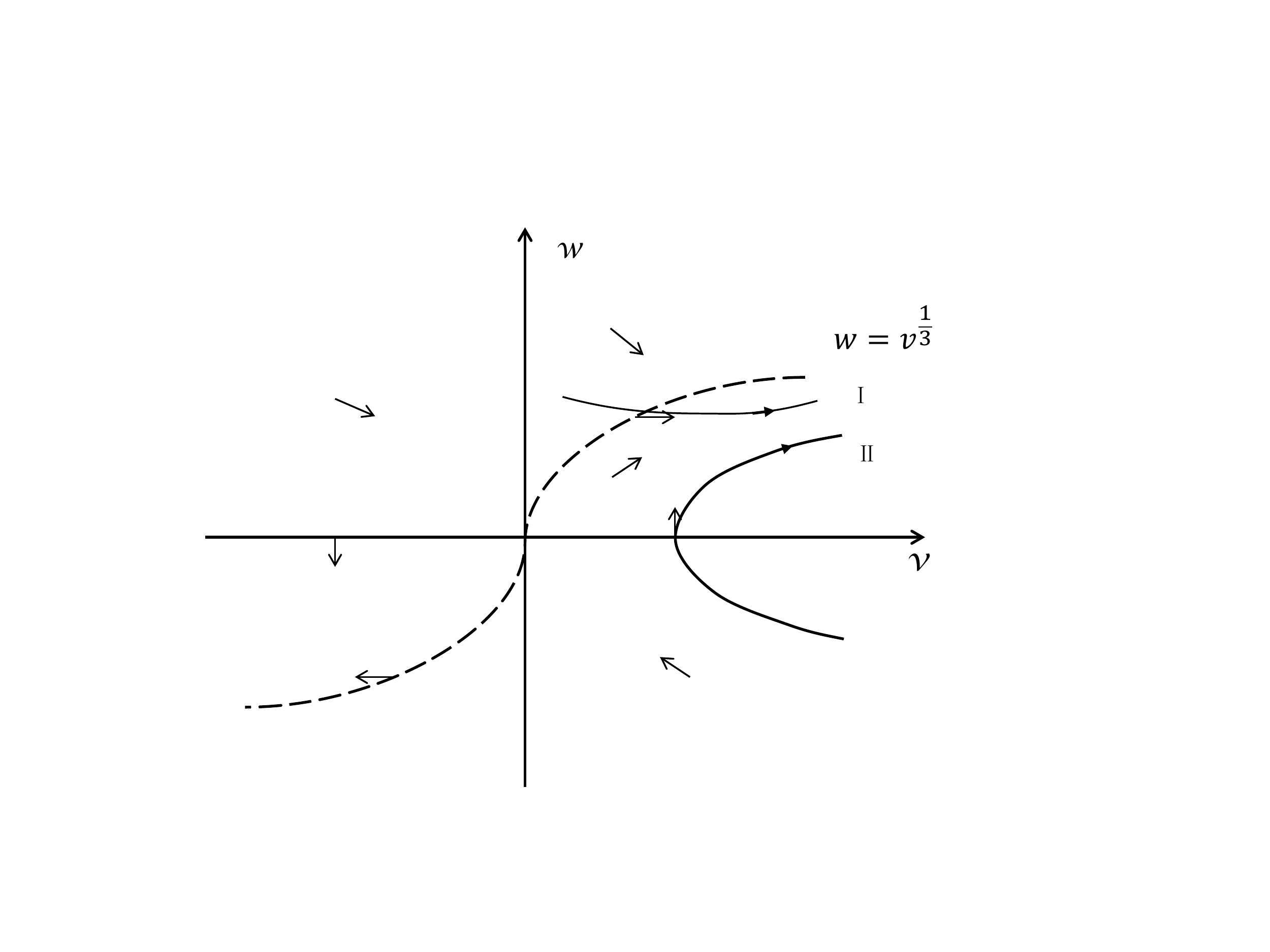}
		\caption{Dynamics of $w-v$ in the proof of Theorem \ref{thm-1-c}}\label{Fig-1-c-w-v}
	\end{minipage}
\begin{minipage}{0.45\textwidth}
		\centering
		\includegraphics[width=1\linewidth, height=0.3\textheight]{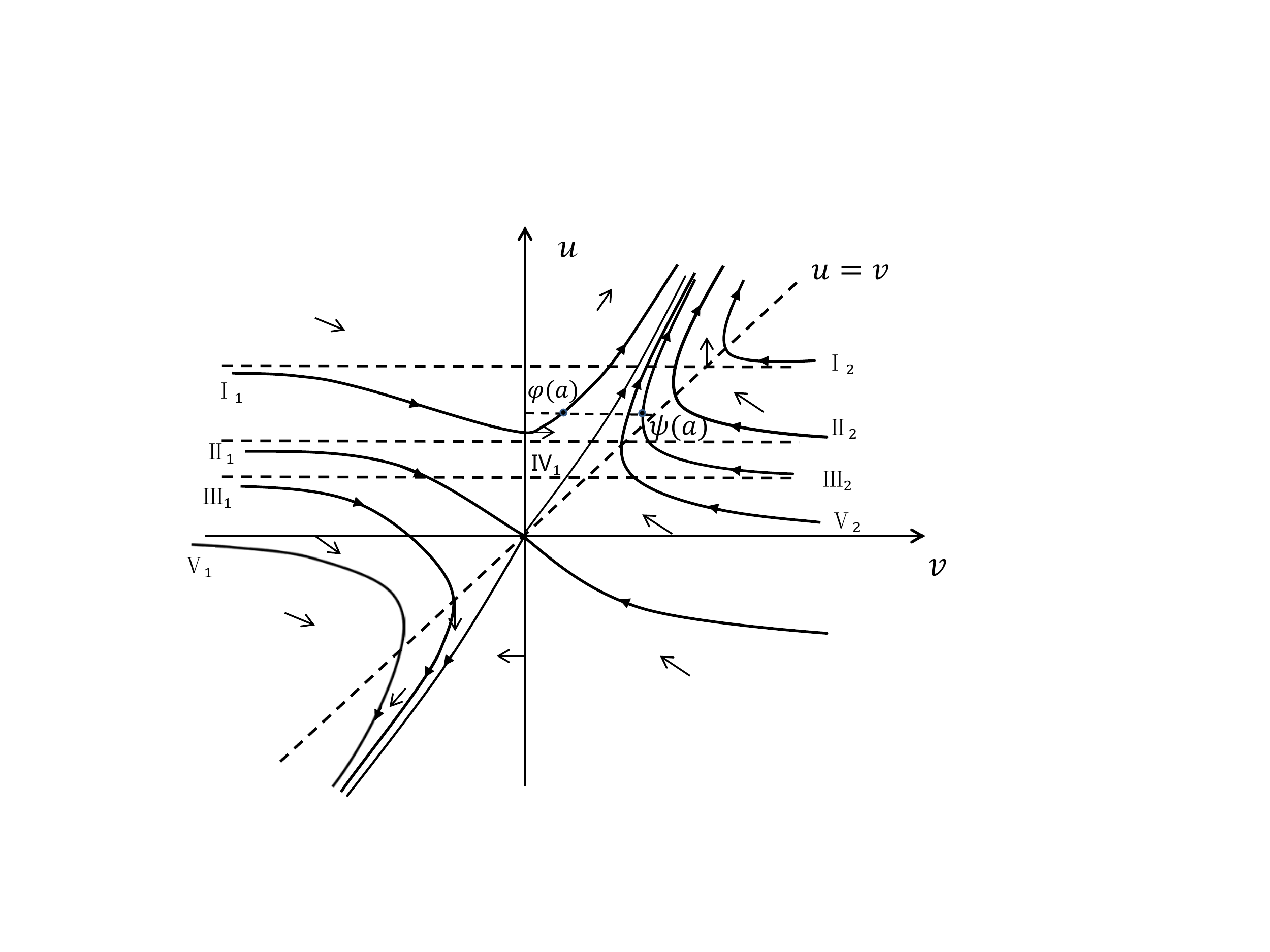}
		\caption{Dynamics of $u-v$ in the proof of Theorem \ref{thm-1-c}}\label{Fig-1-c-u-v}
	\end{minipage}	
\end{figure*}

%\begin{figure}
%  % Requires \usepackage{graphicx}
%  \includegraphics[width=4cm]{Fig-1-c-w-v.pdf}\\
%  \caption{Dynamics of $w-v$ in the proof of Theorem \ref{thm-1-c}}\label{Fig-1-c-w-v}
%\end{figure}

Finally, we show that there are trajectories of \eqref{equ-1-c-u-v} in each connected component of $\R^2$ with the lines $v=0$ and $u-v=0$ deleted that tend to the origin as $x\to -\infty$ or $x\to +\infty$. Consider the map $\varphi$ and $\psi$ from $(0,\epsilon]$ to $[0,\epsilon]$ by sending $a$ to $b=\varphi(a)$ such that $(b,\epsilon)$ is on  the trajectory starting at $(a,0)$ and by sending $a$ to $b=\psi(a)$ such that $(b,\e)$ is on the trajectory starting at $(a,a)$  respectively (see Figure \ref{Fig-1-c-u-v}). It is clear that $\varphi$ and $\psi$ are both topological embedding and $\phi((0,\e])$ and $\psi((0,\e])$ are both open in $[0,\e]$. Therefore $[0,\e]\setminus(\phi((0,\e])\cup\psi((0,\e]))$ is not empty and the trajectory passing through $(\e,b)$ with $b\in[0,\e]\setminus(\phi((0,\e])\cup\psi((0,\e]))$ will tend to the origin as $x\to -\infty$. This proves the claim in the first connected component. The the proof of the claim in other connected components is similar.

%\begin{figure}
%  % Requires \usepackage{graphicx}
%  \includegraphics[width=5cm]{Fig-1-c-u-v.pdf}\\
%  \caption{Dynamics of $u-v$ in the proof of Theorem \ref{thm-1-c}}\label{Fig-1-c-u-v}
%\end{figure}

By the analysis of trajectories of \eqref{equ-1-c-u-v} above, the phase diagram of the dynamical system is as in Figure \ref{Fig-1-c-u-v}. Note that the two trajectories with the same asymptotic behavior as $v\to +\infty$ and $v\to -\infty$ respectively will math up to form a complete curve. So,  the trajectories $I_1$ and $I_2$ in Figure \ref{Fig-1-c-u-v} match up to form the curve of (2), the trajectories $II_1$ and $II_2$ in Figure \ref{Fig-1-c-u-v} match up to form the curve of (3), the trajectories $III_1$ and $III_2$ in Figure \ref{Fig-1-c-u-v} match up to form the curve of (4), and the trajectory $IV$ corresponds to the curve of (5).
\end{proof}
\begin{thm}[Contraction in a single direction]\label{thm-1-d}Solutions of the equation \eqref{equ-1-d} are the following curves:
\begin{enumerate}
\item The lines $x=C$ and $y=0$.
\item The graph of a $T$-periodic function with the function in a periodicity given by
    \begin{equation}
    x-C_2=\left\{\begin{array}{ll}\int_0^{y}\frac{1}{\sqrt{C_1-\frac{1}{2}\xi^4}}d\xi& -\sqrt[4]{2C_1}\leq y\leq \sqrt[4]{2C_1}\\
    \frac{T}{2}-\int_0^y\frac{1}{\sqrt{C_1-\frac{1}{2}\xi^4}}d\xi& 0\leq y\leq \sqrt[4]{2C_1}\\
    -\frac{T}2-\int_0^y\frac{1}{\sqrt{C_1-\frac{1}{2}\xi^4}}d\xi&-\sqrt[4]{2C_1}\leq y\leq 0\end{array}\right.
    \end{equation}
where $T=4\int_0^{\sqrt[4]{2C_1}}\frac{1}{\sqrt{C_1-\frac{1}{2}\xi^4}}d\xi$ and $C_1$ is a positive constant.
\end{enumerate}
\end{thm}
\begin{proof}
Locally written the curve $Y$ in the form of $(x(y),y)$. Then, similarly as in the proof of Theorem \ref{thm-1-a}, we have
\begin{equation}\label{equ-1-b-x_yy}
x_{yy}=y^3x_y^3.
\end{equation}
First of all, $x_y=0$ satisfies the equation. So $x=C$ is a solution of \eqref{equ-1-d}. When $x_y\neq 0$, we have
\begin{equation}
(x_y^{-2})_y=-2y^3.
\end{equation}
So
\begin{equation}
x_y^{-2}=C_1-\frac{1}{2}y^4
\end{equation}
where $C_1$ is a positive constant. Then
\begin{equation}
x(y)=\pm\int_{0}^y\frac{1}{\sqrt{C_1-\frac12\xi^4}}d\xi+C_2
\end{equation}
with $y\in [-\sqrt[4]{2C_1},\sqrt[4]{2C_1}]$. From this, we obtain (2).

On the other hand, when locally written $Y$ in the form of $(x,y(x))$, similarly as in the proof of Theorem \ref{thm-1-a}, we have
\begin{equation}
y_{xx}=-y^3.
\end{equation}
So $y=0$ is a solution of \eqref{equ-1-d}.
\end{proof}
\begin{thm}[Contraction in a direction and translation in another direction]\label{thm-1-e} Up to translations and reflection with respect to the $x$-axis, solutions of \eqref{equ-1-e} are the following curves:
\begin{enumerate}
\item The line $y=0$;
\item Curves formed by the graphs of the functions $f_1$ and $-f_2$. Here $f_1$ and $f_2$ are functions defined on $[0,+\infty)$ oscillating around the $x$-axis satisfying the following properties:

    \begin{enumerate}
    \item $f_1(0)=f_2(0)\geq 0$ and $f_1'(0)=f_2'(0)=+\infty$;
    \item There are three sequences $\{x_n^{(i)}\}_{n=1}^\infty,\{y_n^{(i)}\}_{n=1}^\infty$ and $\{z_n^{(i)}\}_{n=1}^\infty$ in $(0,\infty)$ increasing to $+\infty$ as $n\to +\infty$ for $f_i$ with $i=1,2$ such that
    \begin{enumerate}
    \item $x_n^{(i)}<y_n^{(i)}<z_n^{(i)}<x_{n+1}^{(i)}$ for any $n$ and $i=1,2$;
    \item $x_n^{(i)}$ is a local maximum point of $f_i$ with $f_i(x_n)>0$ when $n$ is odd, and is a local minimum point of $f_i$ with $f_i(x_n)<0$ when $n$ is even;
    \item $y_1^{(i)},y_2^{(i)},\cdots$ are all the inflection points of $f_i$;
    \item $z_1^{(i)},z_2^{(i)},\cdots$ are all the positive zeroes of $f_i$;
    \item $f_i$ is monotone between $x_n^{(i)}$ and $x_{n+1}^{(i)}$;
    \item $\{|f_i(x_n)|\}_{n=1}^\infty$, $\{|f_i(y_n)|=|f'_i(y_n)|\}_{n=1}^\infty$ and $\{|f'_{i}(z_n)|\}_{n=1}^\infty$ decrease to $0$ as $n\to\infty$;
    \item $y^{(i)}_{n}-x^{(i)}_{n}<x^{(i)}_{n+1}-y^{(i)}_n$ and $x^{(i)}_{n+1}-z_{n}^{(i)}<z_{n+1}^{(i)}-x_{n+1}^{(i)}$ for any $n$. Moreover,  $x^{(i)}_{n+1}-y^{(i)}_n$ and $z_{n+1}^{(i)}-x_{n+1}^{(i)}$ tend to $\infty$ as $n\to\infty$.
        \end{enumerate}
    \end{enumerate}
    Moreover, when $f_1(0)=f_2(0)=0$, the curve is symmetric with respect to the $x$-axis (see Figure \ref{Fig-1-e-1}). That is $f_1=f_2$.
\item Graph of a function $f$ on $(-\infty,+\infty)$ oscillating around the $x$-axis (See Figure \ref{Fig-1-e-2}.). More precisely, there are three sequences $\{x_n\}_{n=1}^\infty,\{y_n\}_{n=1}^\infty$ and $\{z_n\}_{n=1}^\infty$ in $(0,\infty)$ increasing to $+\infty$ as $n\to +\infty$ such that
    \begin{enumerate}
    \item $f$ decreases to $-\infty$ as $x\to -\infty$ on $(-\infty,x_1]$.
    \item $x_n<y_n<z_n<x_{n+1}$ for any $n$;
    \item $x_n$ is a local maximum point of $f$ with $f(x_n)>0$ when $n$ is odd, and is a local minimum point of $f$ with $f(x_n)<0$ when $n$ is even;
    \item $y_1,y_2,\cdots$ are all the inflection points of $f$;
    \item $z_1,z_2,\cdots$ are all the zeroes of $f$ greater than $x_1$;
    \item $f$ is monotone between $x_n$ and $x_{n+1}$;
    \item $\{|f(x_n)|\}_{n=1}^\infty$, $\{|f(y_n)|=|f'(y_n)|\}_{n=1}^\infty$ and $\{|f'(z_n)|\}_{n=1}^\infty$ decrease to $0$ as $n\to\infty$;
    \item $y_{n}-x_{n}<x_{n+1}-y_n$ and $x_{n+1}-z_{n}<z_{n+1}-x_{n+1}$ for any $n$. Moreover, $x_{n+1}-y_n$ and $z_{n+1}-x_{n+1}$ tend to $+\infty$ as $n\to\infty$.
        \end{enumerate}

\end{enumerate}
\end{thm}

\begin{figure*}[!htb]
	\centering
	\begin{minipage}{0.45\textwidth}
		\centering
		\includegraphics[width=1.0\linewidth, height=0.2\textheight]{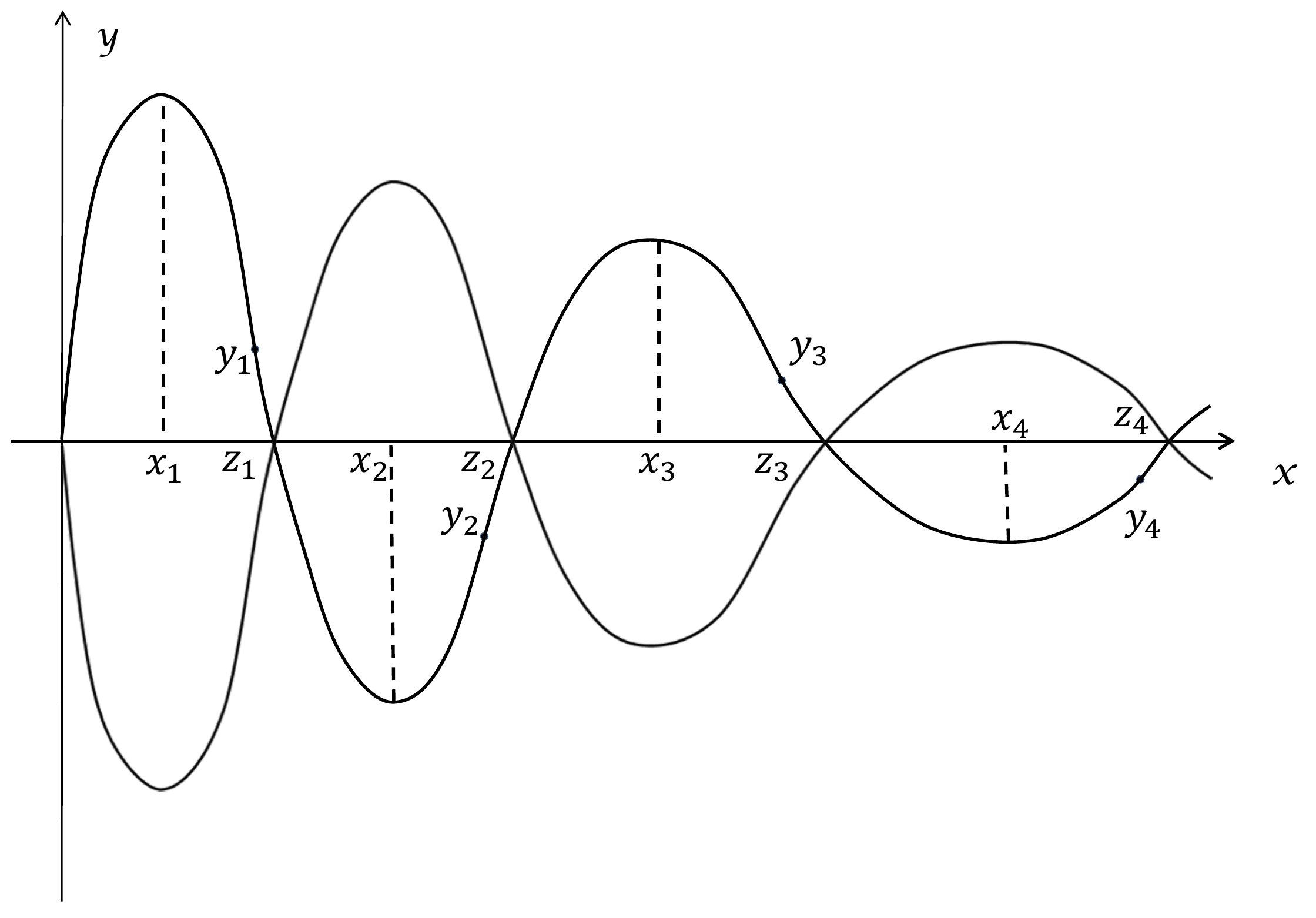}
		\caption{The symmetric curve of (2) in Theorem \ref{thm-1-e}.}\label{Fig-1-e-1}		
	\end{minipage}%
		\begin{minipage}{0.45\textwidth}
		\centering
		\includegraphics[width=1.0\linewidth, height=0.2\textheight]{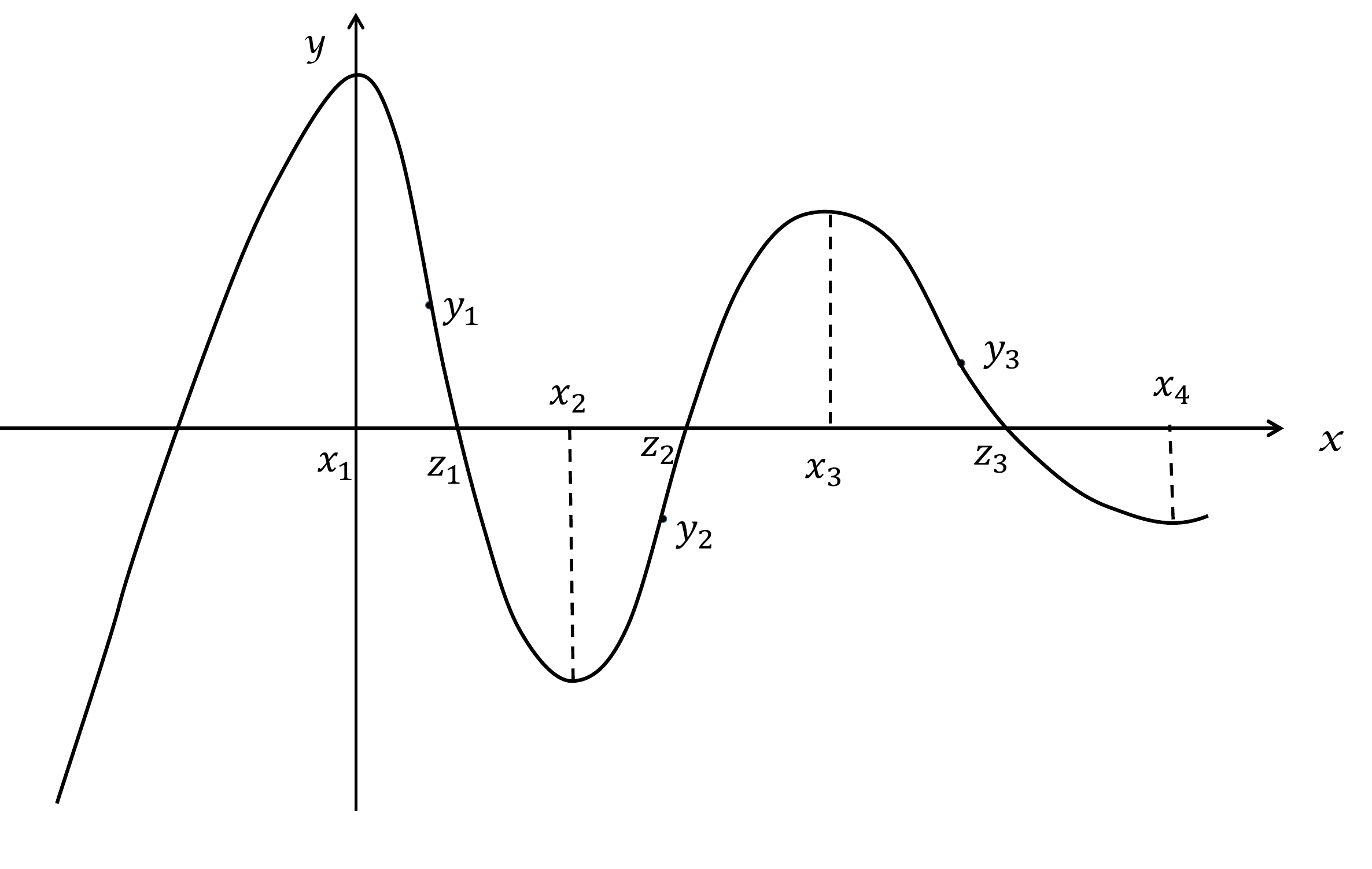}
		\caption{The curve of (3) in Theorem \ref{thm-1-e}.}\label{Fig-1-e-2}
	\end{minipage}
	\end{figure*}
%\begin{figure}
%  % Requires \usepackage{graphicx}
%  \includegraphics[width=5cm]{Fig-1-e-1.pdf}\\
%  \caption{The symmetric curve of (2) in Theorem \ref{thm-1-e}.}\label{Fig-1-e-1}
%\end{figure}
%\begin{figure}
%  % Requires \usepackage{graphicx}
%  \includegraphics[width=5cm]{Fig-1-e-2.pdf}\\
%  \caption{The curve of (3) in Theorem \ref{thm-1-e}.}\label{Fig-1-e-2}
%\end{figure}
\begin{proof} Locally written $Y$ in the form of $(x,y(x))$. Then, similarly as in the proof of Theorem \ref{thm-1-a}, we have
\begin{equation}
y_{xx}=-(y+y_x)^3.
\end{equation}
It is clear that if $y$ is a solution of the equation, then so is $-y$.

Let $u=y$ and $v=y_x$. Then, we have the following planar dynamical system:
\begin{equation}\label{equ-1-e-u-v}
\left\{\begin{array}{l}u_x=v\\v_x=-(u+v)^3.
\end{array}\right.
\end{equation}
The system has a singularity at the origin which corresponds to the solution $y=0$. For convenience of presentation, we will divide the proof into several claims.

\noindent {\bf Claim 1.} As $x$ increasing, a trajectory of \eqref{equ-1-e-u-v} will spiral to the origin. This also implies that $x$ must go to $+\infty$ along any trajectory.

\noindent{\bf Proof of Claim 1.}
Note that, along a trajectory of the dynamical system,
\begin{equation}\label{equ-1-e-decreasing}
\frac{d}{dx}(u^4+2v^2)=-4v^2(u^2+u(u+v)+(u+v)^2)< 0
\end{equation}
where $v\neq 0$. So, as $x$ increasing, a trajectory will tend to the origin. This implies that $x$ must go to $+\infty$ along any trajectory.

Next, we show that any trajectory of \eqref{equ-1-e-u-v} will spiral to the origin as $x\to+\infty$. Suppose this is not true. Without loss of generality, assume that there is a trajectory tending to the origin in the region $R=\{(u,v)\ |\ v<0\ \mbox{and}\ u+v>0\}$. Consider the line $v=-\epsilon$ with $\epsilon$ a small positive number. Then, it is clear that any trajectory can only passing through the line from the RHS to the LHS in the region $R$ (see Figure \ref{Fig-1-e-u-v}). This implies that a trajectory of \eqref{equ-1-e-u-v} can not tend to the origin in the region $R$.

\begin{figure}
  % Requires \usepackage{graphicx}
  \includegraphics[width=0.8\linewidth, height=0.3\textheight]{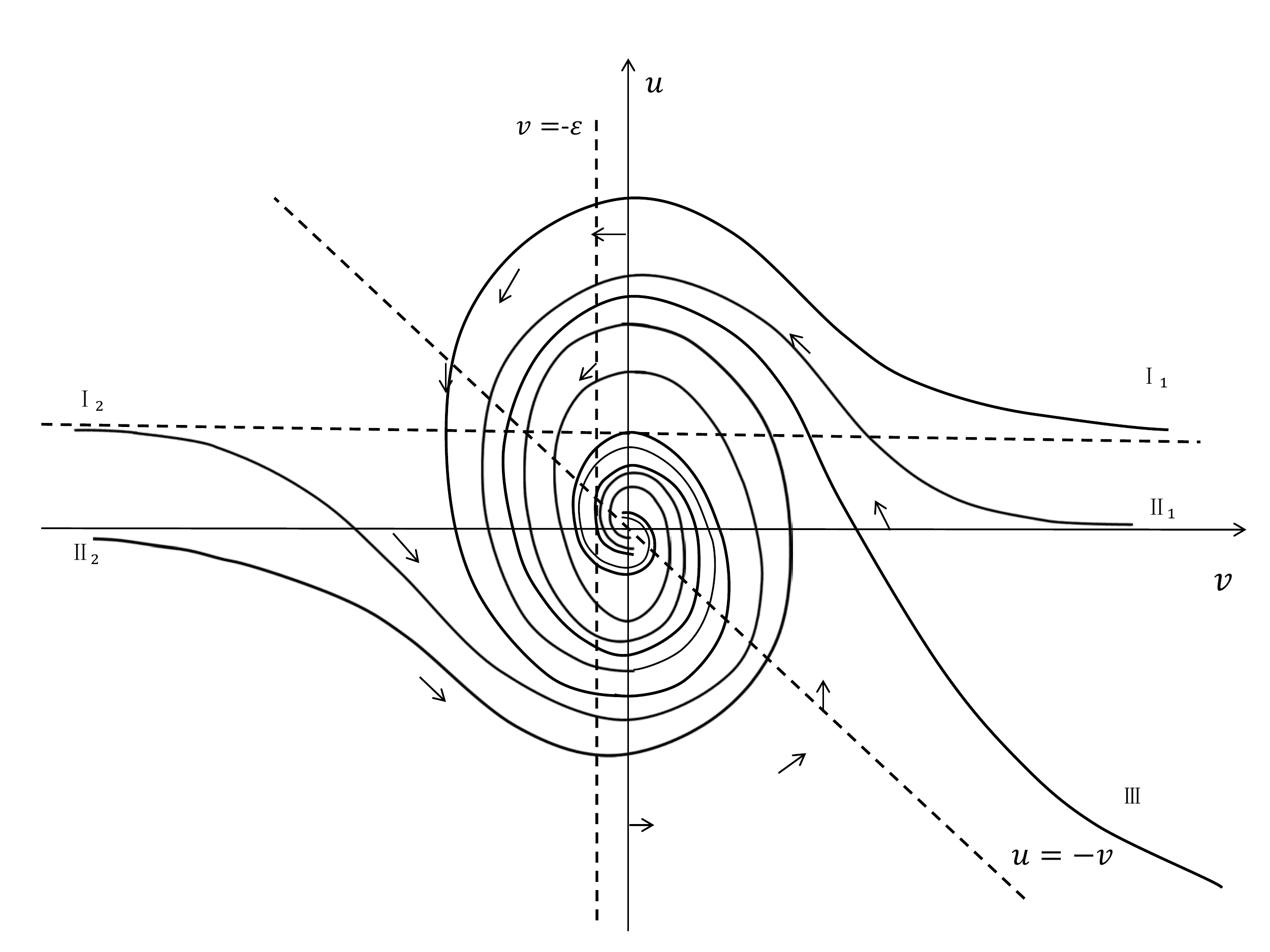}\\
  \caption{Dynamics of $u-v$ in the proof of Theorem \ref{thm-1-e}}\label{Fig-1-e-u-v}
\end{figure}

\noindent{\bf Claim 2.}
Let $\varphi:(-\infty,0)\to (0,+\infty)$ be the map sending $u\in (-\infty,0)$ to $\varphi(u)\in (0,+\infty)$ with $(\varphi(u),0)$  the first intersection point of the trajectory of \eqref{equ-1-e-u-v} passing through $(u,0)$ with the positive axis of $u$. Then $\varphi$ is decreasing and $\varphi\leq 2$.

\noindent{\bf Proof of Claim 2.} Because two different trajectories of \eqref{equ-1-e-u-v} cannot intersect,  $\varphi$ is decreasing. Moreover, for any trajectory with $u(0)>2$ and $v(0)=0$, we want to show that $u(x)>1$ for any $x<0$. Suppose this is not true. Let $x_0<0$ be the first intersection time of the trajectory with the line $u=1$. Then $u(x)>1$ for any $x\in (x_0,0]$. So, by \eqref{equ-1-e-u-v},
\begin{equation}
u(x_0)=u(0)-\int_0^{v(x_0)}\frac{v}{(u+v)^3}dv>2-\int_0^{\infty}\frac{v}{(1+v)^3}dv>1
\end{equation}
which contradicts that $u(x_0)=1$. This completes the proof of Claim 2.

\noindent{\bf Claim 3.} Let $u_m=\sup_{u<0}\varphi(u)$. Then, for any $u_0\geq u_m$, the trajectory of \eqref{equ-1-e-u-v} with $u(0)=u_0$ and $v(0)=0$ will not pass through the line $u+v=0$ for $x<0$.

\noindent{\bf Proof of Claim 3.} We only need to show that any trajectory of \eqref{equ-1-e-u-v} with $(u(0),v(0))$ in the region $\tilde R=\{(u,v)\ |\ v>0, u+v<0\}$ will intersects the negative axis of $u$ for some $x<0$. If this is not true, then $v(x)>v_0>0$ for any $x<0$. Then, it must have
\begin{equation}
\liminf_{u\to -\infty}\left|\frac{dv}{du}\right|\to 0.
\end{equation}
However, by \eqref{equ-1-e-u-v},
\begin{equation}
\frac{dv}{du}=-\frac{(v+u)^3}{v}\to +\infty
\end{equation}
as $u\to -\infty$ which is a contradiction.

\noindent{\bf Claim 4.} Let $\psi:[u_m,+\infty)\to [-\infty,+\infty)$ be such that for any $u_0\geq u_m$, $\psi(u_0)=\inf_{x<0} u(x)$ where $(u(x),v(x))$ is the trajectory with $u(0)=u_0$ and $v(0)=0$. Then, $\psi$ is increasing and surjective.

\noindent{\bf Proof of Claim 4.} By Claim 3, the trajectory $(u(x),v(x))$ will stay in the region $\{(u,v)\ |\ v>0\ \mbox{and}\ u+v>0\}$. So $u$ is decreasing with respect to $v$ along the trajectory by \eqref{equ-1-e-u-v} and $v\to \infty$ as $x$ decreasing. Therefore, $\psi(u_0)=\lim_{v\to\infty}u(v)$ along the trajectory. Then, by that two different trajectories will not intersect, we know that $\psi$ is increasing.

Moreover, to show that $\psi$ is surjective, we only need to show that $\psi(u_m)=-\infty$. Written the curve $Y$ in the form $(x(y),y)$, then similarly as in  Theorem \ref{thm-1-a}, we have
\begin{equation}\label{equ-1-e-x-yy}
x_{yy}=(yx_y+1)^3.
\end{equation}
The equation has a solution $x(y)$ for any given initial data $x(y_0)=0$ and $x_y(y_0)=0$. By \eqref{equ-1-e-x-yy}, $x_{yy}(y_0)=1$. So, $x_y(y)>0$ for $y_0+\epsilon>y>y_0$ with $\e$ small enough. Thus $(y,\frac{dy}{dx})$ of the solution $x(y)$ with $y>y_0$ is part of a trajectory of \eqref{equ-1-e-u-v}, such that $y$ tends to $y_0$ as $x\to 0^+$ and $y_x(x)\to +\infty$ as $x\to 0^+$. By Claim 1, we know that this trajectory will pass through the positive axis of $u$. Therefore $y_0$ is in the image of $\psi$. Because $y_0$ is arbitrary, $\psi(u_m)=-\infty$. This completes the proof of the claim.

\noindent{\bf Claim 5.} Let $(u(x),v(x))$ be a trajectory of \eqref{equ-1-e-u-v} with $u(0)=u_0\geq u_m$ and $v(0)=0$ such that $\psi(u_0)=-\infty$. Then,  $$u(x)\leq -v(x)+2v(x)^\frac{1}3$$ for any $v(x)>1$. This implies that $x\to -\infty$ along the trajectory.

\noindent{\bf Proof of Claim 5.} Suppose that the claim is not true. Then, there is  a $v_0>1$  such that
\begin{equation}
u(v_0)>-v_0+2v_0^\frac13.
\end{equation}
Let $a=-v_0+v_0^\frac13$. Since $u(v)\to-\infty$ as $v\to+\infty$, let $\xi_0>v_0$ be the first time with $u(\xi_0)=a$. Then, for any $v\in (v_0,\xi_0)$, $u(v)> a$. Hence, by \eqref{equ-1-e-u-v},
\begin{equation}
\begin{split}
-v_0+v_0^\frac13=&u(\xi_0)\\
=&u(v_0)-\int_{v_0}^{\xi_0}\frac{v}{(u(v)+v)^3}dv\\
\geq&-v_0+2v_0^\frac13-\int_{v_0}^{\infty}\frac{v}{(a+v)^3}dv\\
=&-v_0-\frac{1}{2}v_0^{-\frac13}+\frac{3}{2}v_0^\frac13.
\end{split}
\end{equation}
which is impossible for $v_0>1$. So, we obtain a contradiction.

Moreover, by \eqref{equ-1-e-u-v},
\begin{equation}
\begin{split}
x=&x_1-\int_{1}^{v(x)}\frac{1}{(u+v)^3}dv\\
\leq& x_1-\frac{1}{8}\int_{1}^{v(x)}\frac{1}{v}dv\\
=&x_1-\frac{1}{8}\ln v(x)
\end{split}
\end{equation}
where $v(x_1)=1$. Since $v\to+\infty$ as $x$ decreasing, $x\to -\infty$ along the trajectory. This completes the proof of the claim.

Note that a trajectory of \eqref{equ-1-e-u-v} intersects the line $v=0$ at extremal points of $y=y(x)$ and intersects the line $u+v=0$ at inflection points of $y=y(x)$, and two trajectories of \eqref{equ-1-e-u-v} with the same limits of $u$ as $v$ tending to $+\infty$ and $-\infty$ respectively will match up to form a complete curve in (2). Moreover, those trajectories of \eqref{equ-1-e-u-v} with $u\to-\infty$ as $v\to+\infty$ correspond to curves in (3) by Claim 5 (see Figure \ref{Fig-1-e-u-v}). So, we have shown (2) and (3) except the last statements which are direct corollaries of Claim 6 and Claim 7 below.

\noindent{\bf Claim 6.} For any $n\geq 1$, $$y_{n}-{x_n}<\frac{A}{\sqrt{|f'(y_n)|}}<x_{n+1}-y_n,$$
where $$A=2^{-\frac34}\int_0^1(1-x^2)^{-\frac34}dx.$$

\noindent{\bf Proof of Claim 6.} Let $w=v_x$. Then, by \eqref{equ-1-e-u-v}, we have the following planar dynamical system:
\begin{equation}\label{equ-1-e-w-v}
\left\{\begin{array}{l}v_x=w\\
w_x=-3w^\frac{2}{3}(v+w).
\end{array}\right.
\end{equation}
By \eqref{equ-1-e-w-v}, along a trajectory,
\begin{equation}\label{equ-1-e-v-mono}
\left(2v^2+w^\frac{4}3\right)_x=-4w^2\leq 0.
\end{equation}
Hence, when $x>y_n$,
\begin{equation}
2v^2+w^\frac43\leq 2v(y_n)^2=2f'(y_n)^2.
\end{equation}
Moreover, when $n$ is odd, $v(y_n)=f'(y_n)<0$ and
\begin{equation}
0<\frac{dv}{dx}\leq \left(2f'(y_n)^2-2v^2\right)^\frac{3}4
\end{equation}
for $x\in (y_n,x_{n+1}]$. Hence
\begin{equation}
\frac{d}{dx}\left(\int_{f'(y_n)}^v\left(2f'(y_n)^2-2\xi^2\right)^{-\frac{3}4}d\xi-x\right)\leq 0
\end{equation}
and
\begin{equation}
x-y_n\geq \int_{f'(y_n)}^{v(x)}\left(2f'(y_n)^2-2\xi^2\right)^{-\frac{3}4}d\xi
\end{equation}
for any $x\in [y_n,x_{n+1}]$. Then,
\begin{equation}\label{equ-1-e-x-y-d}
x_{n+1}-y_{n}>\int_{f'(y_n)}^0\left(2f'(y_n)^2-2\xi^2\right)^{-\frac{3}4}d\xi=\frac{A}{\sqrt{|f'(y_n)|}}
\end{equation}
Similarly, when $n$ is even, then $v(y_n)=f'(y_n)>0$ and
\begin{equation}
0>\frac{dv}{dx}\geq -\left(2f'(y_n)^2-2v^2\right)^\frac{3}4
\end{equation}
for any $x\in (y_n,x_{n+1}]$. So,
\begin{equation}
\frac{d}{dx}\left(\int_v^{f'(y_n)}\left(2f'(y_n)^2-2\xi^2\right)^{-\frac{3}4}d\xi-x\right)\leq0.
\end{equation}
This also gives us \eqref{equ-1-e-x-y-d}.

On the other hand, by \eqref{equ-1-e-v-mono}, for $x\leq y_n$, we have
\begin{equation}
2v^2+w^\frac43\geq 2f'(y_n)^2.
\end{equation}
So, when $n$ is odd,
\begin{equation}
\frac{dv}{dx}\leq -(2f'(y_n)^2-2v^2)^\frac{3}{4}
\end{equation}
and hence
\begin{equation}
\frac{d}{dx}\left(x+\int_{f'(y_n)}^v(2f'(y_n)^2-2\xi^2)^{-\frac{3}{4}}d\xi\right)\leq 0
\end{equation}
for $x\in [x_n,y_n]$. This gives us
\begin{equation}
y_n-x_n\leq \frac{A}{\sqrt{|f'(y_n)|}}.
\end{equation}
The case that $n$ is even can be similarly shown as before. This completes the proof of the claim.

\noindent{\bf Claim 7.} For any $n\geq 1$,
$$x_{n+1}-z_n<\frac{B}{|f(x_{n+1})|}<z_{n+1}-x_{n+1} $$
where
$$B=\sqrt 2\int_0^1\frac{1}{\sqrt{1-x^4}}dx.$$
\noindent{\bf Proof of Claim 7.} The proof is similar with that of Claim 6 by using \eqref{equ-1-e-decreasing}.
\end{proof}
\begin{thm}[Skew steady]\label{thm-1-f} Solutions of the equation \eqref{equ-1-f} are the following curves:
\begin{enumerate}
\item The lines $y=C$;
\item The curves $x=\frac{1}{20}y^5+C_1y+C_2$.
\end{enumerate}
\end{thm}
\begin{proof}
Assume that the curve $Y$ is of the form $(x(y),y)$. Similarly as in the proof of Theorem \ref{thm-1-a}, we have
\begin{equation}\label{equ-1-f-x_yy}
x_{yy}=y^3.
\end{equation}
So $x=\frac{1}{20}y^5+C_1y+C_2$ where $C_1$ and $C_2$ are any constants.

On the other hand, when locally written $Y$ in  the form $(x,y(x))$, similarly as in the proof of Theorem \ref{thm-1-a}, we have
\begin{equation}
y_{xx}=-(yy_x)^3.
\end{equation}
It is clear that $y=C$ is a solution of the equation.

This completes the proof of the Theorem.
\end{proof}
\begin{thm}[Skew steady with translation]\label{thm-1-g} Solutions of the equation \eqref{equ-1-g} are the following curves:
\begin{enumerate}
\item The parabolas: $x=\frac{1}{2}y^2-y+C$.
\item The curves given in the following parametrization form:
\begin{equation}\label{equ-1-g-x}
x(w)=\left\{\begin{array}{ll}\int_w^{\infty}\frac{ y_+(\xi)}{\xi^3-1}d\xi+{(C_1-1)}y_+(w)+\frac{2}{\sqrt 3}\arctan\left(\frac{2w+1}{\sqrt 3}\right)-\frac{\pi}{\sqrt 3}+C_2&1\leq w\leq \infty\\\int^w_{-\infty}\frac{ y_-(\xi)}{1-\xi^3}d\xi+{(C_1-1)}y_-(w)+\frac{2}{\sqrt 3}\arctan\left(\frac{2w+1}{\sqrt 3}\right)
+\frac{\pi}{\sqrt 3}+C_2&w<1\end{array}\right.\\
\end{equation}
and
\begin{equation}\label{equ-1-g-y}
y(w)=\left\{\begin{array}{ll}y_+(w)+C_1&1<w\leq\infty\\
y_-(w)+C_1&w<1\end{array}\right.
\end{equation}
where $y_\pm(w)$ are defined in \eqref{equ-1-g-y_+} and \eqref{equ-1-g-y_-} respectively.
\end{enumerate}
\end{thm}
\begin{proof}
Assume that the curve $Y$ is of the form $(x(y),y)$. Similarly as in the proof of Theorem \ref{thm-1-a}, we have
\begin{equation}\label{equ-1-g-x_yy}
x_{yy}=(y-x_y)^3.
\end{equation}
Let $w=y-x_y$. Then, by \eqref{equ-1-g-x_yy},
\begin{equation}\label{equ-1-g-w_y}
\frac{dw}{dy}=1-w^3.
\end{equation}
First of all, $w=1$ is a solution of \eqref{equ-1-g-w_y}. Then, by the definition of $w$, we obtain (1).

Moreover, when $w>1$, $y$ is decreasing with respect to $w$ and
\begin{equation}
\begin{split}
y(w)=&\int_w^\infty\frac{1}{\xi^3-1}d\xi+C_1=y_+(w)+C_1
\end{split}
\end{equation}
where
\begin{equation}\label{equ-1-g-y_+}
\begin{split}
y_+= -\frac{1}{3}\ln(w-1) +
 \frac16 \ln(1 + w + w^2)+\frac{1}{\sqrt 3}\arctan\left(\frac{2w+1}{\sqrt 3}\right)-\frac{\pi}{2\sqrt 3}.
 \end{split}
\end{equation}
Similarly, when $w<1$, $y$ is increasing with respect to $w$ and
\begin{equation}
\begin{split}
y(w)=&\int_{-\infty}^w\frac{1}{1-\xi^3}d\xi+C_1=y_-(w)+C_1
\end{split}
\end{equation}
where
\begin{equation}\label{equ-1-g-y_-}
\begin{split}
y_-=&- \frac{1}{3}\ln(1-w) +
 \frac16 \ln(1 + w + w^2)+\frac{1}{\sqrt 3}\arctan\left(\frac{2w+1}{\sqrt 3}\right)+\frac{\pi}{2\sqrt 3}.
 \end{split}
\end{equation}
Moreover, by the definition of $w$, we have $x_y=y-w$. So, by \eqref{equ-1-g-w_y},
\begin{equation}\label{equ-1-g-x_w}
\frac{dx}{dw}=\frac{dx}{dy}\frac{dy}{dw}=\frac{ y-w}{1-w^3}={\frac{y_{\pm}(w)}{1-w^3}+\frac{C_1-1}{1-w^3}+\frac{1}{1+w+w^2}}.
\end{equation}
So,
\begin{equation}
x=\left\{\begin{array}{ll}\int_w^{\infty}\frac{ y_+(\xi)}{\xi^3-1}d\xi+{(C_1-1)}y_++\frac{2}{\sqrt 3}\arctan\left(\frac{2w+1}{\sqrt 3}\right)-\frac{\pi}{\sqrt 3}+C_2&w>1\\\int^w_{-\infty}\frac{ y_-(\xi)}{1-\xi^3}d\xi+{(C_1-1)}y_-+\frac{2}{\sqrt 3}\arctan\left(\frac{2w+1}{\sqrt 3}\right)+\frac{\pi}{\sqrt 3}+C_2&w<1.
\end{array}\right.
\end{equation}
Note that $y_+(\infty)=y_-(-\infty)=0$. So, the two parts of $x$ and $y$ with $w>1$ and $w<1$ will match up to form a complete curve at $w=\pm\infty$. From this, we obtain (2). This completes the proof of the Theorem.
\end{proof}

\end{document}